\documentclass[10pt]{article}
 
 \usepackage[utf8]{inputenc}
 \usepackage{amsthm}
 \usepackage{enumitem}
 \usepackage{hyperref}
 \usepackage{amssymb}
 \usepackage{amsmath}
 \usepackage{dsfont}
 \usepackage{mathtools}
 \usepackage{ mathrsfs }
 \usepackage{comment}
 \usepackage[left=3.5cm,right=3.5cm,top=3cm,bottom=2.5cm]{geometry}

 \newtheorem{thm}{Theorem}
 \newtheorem{prop}{Proposition}
 \newtheorem{lem}{Lemma}

 \newtheorem{rem}{Remark}

\title{Intrinsic area near the origin for self-similar growth-fragmentations and related random surfaces}
\author{François G. Ged\footnote{Insitut für Mathematik, Universität Zürich, Switzerland. E-mail: francois.ged@math.uzh.ch}}
\date{}
\newcommand{\intervalleff}[2]{\left[#1\,,#2\right]}
\newcommand{\intervallefo}[2]{\left[#1\,,#2\right)}
\newcommand{\intervalleof}[2]{\left(#1\,,#2\right]}
\newcommand{\intervalleoo}[2]{\left(#1\,,#2\right)}

\begin{document}
	\maketitle
	\vspace{-1.1cm}
	
	\begin{abstract}
		We study the behaviour of a natural measure defined on the leaves of the genealogical tree of some branching processes, namely self-similar growth-fragmentation processes.
		Each particle, or cell, is attributed a positive mass that evolves in continuous time according to a positive self-similar Markov process and gives birth to children at negative jumps events.
		We are interested in the asymptotics of the mass of the ball centered at the root, as its radius decreases to $0$.
		We obtain the almost sure behaviour of this mass when the Eve cell starts with a strictly positive size.
		This differs from the situation where the Eve cell grows indefinitely from size 0.
		In this case, we show that, when properly rescaled, the mass of the ball converges in distribution towards a non-degenerate random variable.
		We then derive bounds describing the almost sure behaviour of the rescaled mass.
		Those results are applied to certain random surfaces, exploiting the connection between growth-fragmentations and random planar maps obtained in \cite{BBCK16}.
		This allows us to extend a result of Le Gall \cite{L17} on the volume of a free Brownian disk close to its boundary, to a larger family of stable disks.
		The upper bound of the mass of a typical ball in the Brownian map is refined, and we obtain a lower bound as well.\vspace{0.5cm}
		
		

	\end{abstract}

	\section{Introduction}
	Growth-fragmentation processes form a family of continuous time branching processes that have been introduced by Bertoin \cite{B17}.
	They model particle systems without interaction where each particle is described by a positive real number that corresponds to its mass (or size), evolving by growing and splitting, with rates that can depend on its current mass.
	Note that these processes differ from (pure) fragmentation processes \cite{B01,B04}, for which growth is not allowed.
	The fragmentations are binary and when a particle splits, its mass is instantaneously randomly distributed among the two resulting fragments.
	A \textit{self-similar growth-fragmentation} $(\mathbf{X}(t))_{t\geq 0}$ is a growth-fragmentation where every particle evolves according to a driving self-similar Markov process.
	Self-similarity refers to the property that the evolution of a particle of size $x>0$ is a scaling transformation of that of a particle of unit size, depending on some index $\alpha\in\mathbb{R}$.
	The law of $\mathbf{X}$ is characterized by $\alpha$ and its \textit{cumulant function} $\kappa:\mathbb{R}\intervalleof{-\infty}{\infty}$, both depending on the driving process.
	
	In this paper we are interested in the cases where $\alpha<0$ and $\kappa$ has two positive root $\omega_-<\omega_+$.
	Under some further conditions on $\alpha$ and $\kappa$, the growth-fragmentation yields a natural measure on the genealogical tree of the branching process seen as a metric space, namely the \textit{intrinsic area measure}.
	Denoting $A(t)$ the intrinsic area of the ball of radius $t$ centered at the root of the tree, one can investigate the regularity of the Stieltjes measure $\mathrm{d}A(t)$.
	This has been studied in \cite{G18}, in which Theorem 1 shows that if $\alpha>-\omega_-$, then $\mathrm{d}A(t)$ is absolutely continuous whereas it is singular when $\alpha\leq -\omega_-$.
	In this paper, we consider the absolutely continuous case.
	A noteworthy fact is that the density is null at $t=0$, meaning that the dissipation of the area occurs as small particles are about to vanish.
	We study further properties of $A$.
	
	Our first result is Theorem \ref{Theorem area GF}, that determines the almost sure behaviour of $A(\epsilon)$ as $\epsilon\to 0+$, when $\mathbf{X}$ starts from a single particle of positive size $x$.
	Namely, there exists a regularly varying function $f:\mathbb{R}_+\to\mathbb{R}_+$ explicitely given in terms of the characteristics of the driving process, such that $A(\epsilon)/f(\epsilon)\to 1$ as $\epsilon\to 0+$.
	
	It is possible to tilt the law of $\mathbf{X}$ such that the Eve cell starts from size 0.
	Indeed, in \cite{BBCK16}, the authors introduced a new probability measure, under which $\mathbf{X}$ initiates from a distinguished particle that grows indefinitely from the initial size 0, and evolves differently from the others which eventually die out.
	Let denote $A^+$ the analog of $A$ under this new measure.
	In this case, Theorem \ref{Theorem area GF} does not apply.
	However, Proposition \ref{Theorem stationary area from 0} shows that, $t\mapsto e^{\omega t/\alpha}A^+(e^t)$ defines a stationary process.
	In particular, $t^{\omega_-/\alpha}A^+(t)$ has a limit in distribution as $t\to 0+$, but has no almost sure limit.
	We prove in Propositions \ref{Theorem upper envelope A(t)} and \ref{Proposition lower envelope A(t)} that, almost surely, $A^+(t)$ deviates from $t^{\omega_-/|\alpha|}$ by at most a power of $|\log(t)|$.
	
	One of the motivations of the present work is that these branching processes turn out to be geometrically connected to some random surfaces.
	The Brownian map is a random surface homeomorphic to the two-dimensional sphere that appears as the Gromov-Hausdorff scaling limit when $n\to\infty$ of uniformly distributed $q$-angulations with $n$ faces of the sphere \cite{L13,M13}.
	Similarly, Brownian disks are random compact metric spaces homeomorphic to the unit disk of $\mathbb{R}^2$, obtained as scaling limits of random planar maps with a boundary \cite{BM17}. 
	In \cite{CL14}, the Brownian plane, locally isomorphic to the Brownian map, is obtained as scaling limits of the UIPQ and uniform quadrangulations. Moreover it can also be seen as the Gromov-Hausdorff tangent cone of the Brownian map at its root.
	
	In \cite{BCK15}, it was shown that the collection of perimeters of the holes observed when slicing Boltzmann triangulations at all heights converges, when properly rescaled, towards a particular self-similar growth-fragmentation.
	We also mention Theorems 3 and 23 of Le Gall and Riera \cite{LR18}, which show that, when slicing directly the free Brownian disk, the holes' perimeters are described by the same growth-fragmentation as in \cite{BBCK16} (see also \cite{MS15} Section 4).
	When $\mathbf{X}$ starts from a single cell of size 0 that grows indefinitely, the geometrical connection corresponds this time to the holes in a sliced discrete approximation of the Brownian plane.
	
	There is actually a broader family of continuum random surfaces, with different scaling exponents, that arise as limiting objects of Boltzmann random planar maps. They are known as \textit{stable maps} and were first obtained by Le Gall and Miermont \cite[and references therein]{LM11,LM10}, by considering very specific distributions for the degree of a typical face, that have infinite variance (see also \cite{M18} for the same scaling limits under relaxed hypotheses).
	However, only the Brownian map has been characterized yet, in the sense that for other stable maps, the uniqueness of the scaling limit is not yet proven.
	
	Similarly as in the Brownian case, the authors in \cite{BBCK16} were able to extend the geometrical connection previously mentioned, to holes' perimeters in discrete approximations of stable disks and plane (in the so-called \textit{dilute case}) and a specific family of self-similar growth-fragmentations.
	
	There is a natural way of measuring the "size" of these stable surfaces, namely, the so-called \textit{intrinsic volume measure}, which can be constructed as the scaling limit of the number of vertices (or faces) in the approximation by discrete random maps.
	This measure corresponds to the intrinsic area measure in the related growth-fragmentation; we shall rather use this name since we consider planar objects.
	
	Besides being aesthetic, this connection has already been fruitful in both directions: for instance, it allowed the authors in \cite{BBCK16} to use results on discrete random planar maps to determine the law of the total intrinsic area of the related growth-fragmentations.
	On the other hand, it was known that the intrinsic area of the Brownian map cannot be derived as the length of the perimeters of the holes and the height, as for smooth surfaces, since the latter defines a measure that is not locally finite.
	From the analysis of growth-fragmentations, it has been shown in \cite{G18} that the intrinsic area of the Brownian map can be written as the integral against the Lebesgue measure of some function of the perimeters of the holes.
	
	Our results in this paper apply to stable surfaces.
	In particular, Thanks to Theorem \ref{Theorem area GF}, we retrieve Theorem 3 of Le Gall \cite{L17}, which shows that, in the case of the free Brownian disk with boundary length $x>0$, it holds that $\epsilon^{-2}A(\epsilon)\to x$ almost surely, as $\epsilon\to 0+$ (see \eqref{Equation Le Gall}).
	More generally, we obtain the analogue for other stable disks, with different exponents. 
	
	Applying Proposition \ref{Theorem upper envelope A(t)} improves the upper bound that was known for the almost sure behaviour of the area of the ball of radius $\epsilon$ around the root of the Brownian map (i.e. $A^+(\epsilon)$): the previous bound was $\epsilon^{4-\delta}$ for $\delta>0$ arbitrary small, we obtain $\epsilon^{4}|\log(\epsilon)|^{1+\delta}$.
	We moreover obtain the lower bound $\epsilon^4|\log(\epsilon)|^{-q}$ for any $q>6$, thanks to Proposition \ref{Proposition lower envelope A(t)}.
	Again, the analogue holds for stable maps.
	
	The organisation of the paper is the following.
	In Section 2, we formally introduce the growth-fragmentations setting.
	This includes definitions of positive self-similar Markov processes and Lamperti's transformation, the intrinsic area measure of a growth-fragmentation, as well as the two spinal decompositions introduced in \cite{BBCK16} that are central throughout this work.
	We also establish an important Markov-branching property of $A$ in Lemma \ref{Lemma Markov-branching property of A}, which, roughly speaking, reduces the study of $A$ to that of a single self-similar Markov process.

	We state our results in Section 3.
	Theorem \ref{Theorem area GF} concerns the almost sure behaviour of $A(\epsilon)$ as $\epsilon\to 0+$, when $\mathbf{X}$ starts from a typical cell with positive initial size.
	Our second result is Proposition \ref{Theorem upper envelope A(t)} which provides an upper bound for the almost sure behaviour of $A^+(\epsilon)$ when the initial cell starts from $0$ and behaves differently from the others, with indefinite growth.
	A lower bound is then given in Proposition \ref{Proposition lower envelope A(t)}.
	
	We prove Theorem \ref{Theorem area GF} in Section 4.
	The first subsection looks at the expectation of $A(\epsilon)$.
	In the second subsection, we introduce a useful martingale to control the fluctuations of $A(\epsilon)$ around its expectation and prove Theorem \ref{Theorem area GF}.  
	
	Section 5 is devoted to the proof of Propositions \ref{Theorem upper envelope A(t)} and \ref{Proposition lower envelope A(t)}, that are respectively upper bound and lower bound for $A^+(\epsilon)$ as $\epsilon\to 0+$.
	After having established Proposition \ref{Theorem stationary area from 0}, we prove Proposition \ref{Theorem upper envelope A(t)} in the first subsection.
	We then turn our attention to Proposition \ref{Proposition lower envelope A(t)} in the second subsection.
	The arguments are different and the proof is more involved.
	
	We conclude this paper with Section 6, in which we apply our results to stable surfaces.

	\section{Self-similar growth-fragmentations}\label{Section self-siilar growth-fragmentations}
	
	\paragraph{The law of a typical cell.}
	For all $x>0$, let $\mathbb{P}_x$ denote the law of a positive self-similar Markov process $X=(X(t))_{t\geq 0}$ starting from $x$ and absorbed at 0.
	Self-similarity refers to the property that under $\mathbb{P}_x$, the law of $(X(t))_{t\geq 0}$ is the same as that of $(xX(tx^\alpha))_{t\geq 0}$ under $\mathbb{P}_1$, for some real index $\alpha$.
	We shall always assume here that $\alpha<0$.
	For all $t\geq 0$, let introduce the time-change
	\begin{align}\label{Equation Lamperti tau_t}
		\tau_t:=\int_0^tX(s)^\alpha\mathrm{d}s.
	\end{align}
	The well-known Lamperti's transformation states that there exists a unique L\'evy process $\xi=(\xi(t))_{t\geq 0}$ such that 
	\begin{align*}
		X(t)=\exp\left(\xi(\tau_t)\right),\quad \forall t\geq 0.
	\end{align*}
	Equivalently to \eqref{Equation Lamperti tau_t}, one can write
	\begin{align*}
		\tau_t=\inf\left\{s\geq 0:\int_0^s\exp(-\alpha\xi(u))\mathrm{d}u\geq t\right\}.
	\end{align*}
	Let $\Lambda$ be the L\'evy measure of $\xi$ and assume that $\int_{\intervalleoo{1}{\infty}}e^{y}\Lambda(\mathrm{d}y)<\infty$.
	We thus have, at least for $q\in\intervalleff{0}{1}$, that $\mathbb{E}(\exp(q\xi(t)))=\exp(t\psi(q))$ for all $t\geq 0$, with  
	\begin{align}\label{Equation Laplace exponent psi of xi}
		\psi:q\mapsto bq+\frac{\sigma^2}{2}q^2+\int_{\mathbb{R}}\left(e^{qy}-1+q(1-e^y)\right)\Lambda(\mathrm{d}y),\quad q\geq 0,
	\end{align}
	where $b,\sigma^2\geq 0$.
	We point out that this is not the traditional L\'evy-Khintchin formula for the Laplace exponent $\psi$ of a L\'evy process, in which $(1-e^y)$ above is usually replaced by $y\mathds{1}_{\left\{|y|<1\right\}}$, but \eqref{Equation Laplace exponent psi of xi} is more convenient for our purposes.
	We assume throughout this paper that $\psi'(0)<0$. It is known that this entails that $\xi$ drifts to $-\infty$ almost surely, and in particular that the absorption time in $0$ of $X$, namely $\int_0^\infty\exp(-\alpha\xi(t))\mathrm{d}t$, is finite almost surely.
	
	\paragraph{The cell-system.}
	We briefly recall Bertoin's construction of the cell-system as in \cite{B17}.
	We use the classical Ulam-Harris-Neveu notation to label the particles of the branching process.
	Let $\mathbb{U}:=\bigcup_{n\geq 0}\mathbb{N}^n$ and let $\partial\mathbb{U}$ be the set of infinite sequences of natural integers.
	We shall also denote $\overline{\mathbb{U}}:=\mathbb{U}\cup\partial\mathbb{U}$.
	If $u\in\mathbb{N}^n$ for some $n\in\mathbb{N}$, we write $u(k)\in\mathbb{N}^k$ its ancestor at generation $k\leq n$ and $|u|=n$ its generation. 
	
	The process starts with a single cell that we call the Eve cell, indexed by $\emptyset$ and size at any time $t\geq 0$ denoted by $\chi_{\emptyset}(t)$.
	It evolves in time, starting from its birthtime $b_\emptyset:=0$ according to $\mathbb{P}_1$ until its absorption time $\zeta_\emptyset$ at 0. In this context, $\zeta_{\emptyset}$ is called the \textit{lifetime} of $\chi_\emptyset$.
	
	Since $\chi_\emptyset$ converges to 0 almost surely, it is possible to rank all its negative jumps in the decreasing order of the absolute values of their sizes.
	If $\left\{(b_i,\Delta_i);i\geq 1\right\}$ is the collection of times and sizes of these negative jumps ranked in this way, then for each $i\geq 1$, we start at time $b_i$ a new positive self-similar Markov process $\chi_i$ under the law $\mathbb{P}_{|\Delta_i|}$, independently of every other cell. We denote its lifetime $\zeta_i$, that is $\chi_i(t)$ is the size of the cell labelled $i$ if $b_i\leq t<b_i+\zeta_i$, and is sent to some cemetery state $\partial$ otherwise.
	In this manner, we construct recursively the whole cell-system indexed by $\mathbb{U}$, and we denote $\mathcal{P}_1$ its law. More generally, we can construct such a cell-system with the Eve cell starting from a size $x>0$, its law is then denoted $\mathcal{P}_x$.

	\paragraph{The growth-fragmentation.}
	The growth-fragmentation process $(\mathbf{X}(t))_{t\geq 0}$ induced by the cell-system is the process following the collection of particles' sizes in time, forgetting about their genealogy, that is
	\begin{align*}
		\mathbf{X}(t):=\left\{\!\!\left\{\chi_u(t-b_u):u\in\mathbb{U},b_u\leq t<b_u+\zeta_u\right\}\!\!\right\},
	\end{align*}
	where the elements are repeated according to their multiplicity.
	
	The law of $\mathbf{X}$ is characterized by $\alpha$ and a particular function $\kappa:\mathbb{R}_+\to\mathbb{R}\cup\{\infty\}$ called the \textit{cumulant function} of $\mathbf{X}$ (see \cite{S17}), which is defined as
	\begin{align*}
		\kappa(q):=\psi(q)+\int_{\intervalleoo{-\infty}{0}}(1-e^y)^q\Lambda(\mathrm{d}y).
	\end{align*}
	The importance of $\kappa$ for the study of self-similar growth-fragmentations comes from the fact, taken from \cite{B17} Lemma 3, that $\mathbb{E}_x\left(\sum_{i\geq 1}\chi_i(0)^{q}\right)=x(1-\kappa(q)/\psi(q))$, when it makes sense.
	This means that when $\kappa(q)=0$, the initial size of a cell is equal to the expectation of the sum of the sizes raised to the power $q$ of its first generation children at their birthtime.
	For this reason, we assume throughout this work that the \textit{Cram\'er hypothesis} holds, that is there exists $\omega_->0$ such that
	\begin{align}\label{Equation Cramer hypothesis}
		\kappa(\omega_-)=0,\quad-\infty<\kappa'(\omega_-)<0.
	\end{align}
	We suppose that $\kappa$ has a second root $\omega_+>\omega_-$ and is finite in a right neighbourhood of $\omega_+$, which by convexity of $\kappa$ implies $0<\kappa'(\omega_+)<\infty$.
	
	\paragraph{Intrinsic area measure.}
	Thanks to \eqref{Equation Cramer hypothesis}, the process
	\begin{align*}
		\mathcal{M}(n):=\sum_{|u|=n}\chi_u(0)^{\omega_-},\quad n\geq 0
	\end{align*}
	is a uniformly integrable martingale, see Lemma 2.4 in \cite{BBCK16}.
	Its terminal value $\mathcal{M}:=\lim_{n\to\infty}\mathcal{M}(n)$ is called the \textit{intrinsic area} of $\mathbf{X}$ and we point out that, by construction, it does not depend on $\alpha$.
	
	Endowed with the distance $d(\ell,\ell'):=\exp(-\sup\{n\geq 0:\ell(n)=\ell'(n)\})$, $\partial\mathbb{U}$ is a complete metric space.
	The \textit{intrinsic area measure} $\mathcal{A}$ is then the unique measure on $\partial\mathbb{U}$ such that the mass of the subsets of leaves having ancestor $u\in\mathbb{N}^n$ at generation $n\geq 0$ satisfies
	\begin{align*}
		\mathcal{A}\left(\left\{\ell\in\partial\mathbb{U}:\ell(n)=u\right\}\right)=\lim_{k\to\infty}\sum_{|v|=k}\chi_{uv}(0)^{\omega_-}.
	\end{align*}
	Note that the total mass of $\mathcal{A}$ is $\mathcal{M}$.
	Denoting by $\zeta_\ell$ the height of $\ell\in\partial\mathbb{U}$, that is the time at which the lineage of this leaf goes extinct in terms of the growth-fragmentation's time, we define the area of the ball with radius $t\in\mathbb{R}_+$ centered at the origin by
	\begin{align*}
		A:t\mapsto\mathcal{A}\left(\left\{\ell\in\partial\mathbb{U}:\zeta_\ell\leq t\right\}\right).
	\end{align*}
	The increasing function $A$ will be the object of interest of the rest of the paper.
	It satisfies the following Markov-branching type property that we shall use all along this work.
	
	\begin{lem}\label{Lemma Markov-branching property of A}
		For all $t\geq 0$, it holds that
		\begin{align*}
		A(t)=\sum_{s\leq t}|\Delta_- \chi_{\emptyset}(s)|^{\omega_-}A_s\left((t-s)|\Delta_- \chi_\emptyset(s)|^{\alpha}\right),
	\end{align*}
	where $\Delta_-\chi_\emptyset(s):=\min\left\{0, \chi_\emptyset(s)-\chi_\emptyset(s-)\right\}$ (i.e. we only consider the negative jumps) and the $A_s$'s are i.i.d. copies of $A$ under $\mathcal{P}_1$, independent from $\chi_\emptyset$.
	\end{lem}
	
	\begin{proof}
		We first introduce a notation.
		If $\Delta_-\chi_\emptyset(s)\neq 0$ for some $s>0$, i.e. if the Eve cell gives birth at time $s$, we denote $\mathbb{U}_s$ the subtree generated by the newborn cell.
		We write
		\begin{align*}
			A(t)
			&=\mathcal{A}\left(\left\{\ell\in\partial\mathbb{U}:\zeta_\ell\leq t\right\}\right)
			=\sum_{s\leq t}\mathds{1}_{\left\{\Delta_-\chi_{\emptyset}(s)\neq 0\right\}}\mathcal{A}\left(\left\{\ell\in\partial\mathbb{U}_s:\zeta_\ell\leq t\right\}\right).
		\end{align*}
		Thanks to \eqref{Equation Cramer hypothesis}, Lemma 3.2 of \cite{BBCK16} entails that the areas in the above sum are independent and that for all $s>0$ the conditional distribution of $\mathcal{A}\left(\left\{\ell\in\partial\mathbb{U}_s:\zeta_\ell\leq t\right\}\right)$ given $\chi_{\emptyset}$ is that of $A(t-s)$ under $\mathcal{P}_{|\Delta_-\chi_{\emptyset}(s)|}$, which by self-similarity is identical to that of $|\Delta_-\chi_{\emptyset}(s)|^{\omega_-}A((t-s)|\Delta_-\chi_{\emptyset}(s)|^{\alpha})$ under $\mathcal{P}_1$. (The time shift $-s$ comes from the fact that the root of $\mathbb{U}_s$ is at height $s$ in $\mathbb{U}$.)
	\end{proof}
	
	We shall sometimes use Lemma \ref{Lemma Markov-branching property of A} replacing $\chi_\emptyset(s)$ by $X(s)$, or equivalently $\exp(\xi_{\tau_s})$.
	
	\paragraph{Choosing a spine according to the intrinsic area.}
	A classical tool in the study of branching processes is the so called \textit{spinal decomposition} (see \cite{S15}).
	In \cite{BBCK16}, the authors introduced one related to the intrinsic area (we refer to their work for proofs and more details).
	 Let $\mathcal{P}^-_x$ be the joint law of a cell system $(\chi_u:u\in\mathbb{U})$ starting from an Eve cell $\chi_\emptyset$ with initial size $x>0$ and a distinguished leaf $\widehat{\ell}\in\partial\mathbb{U}$, such that the law of $(\chi_u:u\in\mathbb{U})$ under $\mathcal{P}_x^-$ is absolutely continuous with respect to its law under $\mathcal{P}_x$, with density $x^{-\omega_-}\mathcal{M}$, and $\widehat{\ell}$ has conditional law $\mathcal{A}(\cdot)/\mathcal{M}$ given the cell system.
	
	The \textit{spine} is the process following in time the size of the ancestor of $\widehat{\ell}$.
	Up to modifying the genealogy, we can consider the spine to be the Eve cell of the cell-system, the law of the growth-fragmentation remains unchanged (see \cite{BBCK16} Section 4.2).
	Whereas the spine has a tilted distribution, its daughters generate, given the spine, independent growth-fragmentations with laws $(\mathcal{P}_x)_{x>0}$ according to their initial sizes.	
	Its law is that of a positive self-similar Markov process $Y^-=(Y^-(t))_{t\geq 0}$ with same index $\alpha$.
	Let $\mathbb{P}_x^-$ be the law of $Y^-$ starting from $x>0$.
	As for $X$, we have that $(Y^-(t))_{t\geq 0}=(\exp(\eta^-(\tau_t)))_{t\geq 0}$, where $\eta^-$ is a L\'evy process and $(\tau_t)_{t\geq 0}$ the associated time-change from the Lamperti's transformation. (It will always be clear from the context whether $\tau_t$ denotes the time-change associated with $Y^-$ or $X$.)
	Recall that $\Lambda$ is the L\'evy measure of $\xi$.
	The L\'evy measure of $\eta^-$, that we call $\Pi^-$, is then given by
	\begin{align}\label{Equation Pi Levy measure of eta^-}
		\Pi^-(\mathrm{d}y)=e^{\omega_-y}(\Lambda+\widetilde{\Lambda})(\mathrm{d}y),
	\end{align}
	where $\widetilde{\Lambda}$ is the push-forward of $\Lambda$ by the mapping $x\mapsto\log(1-e^x)\mathds{1}_{\{x<0\}}$.
	The Laplace exponent\footnote{See the argument of \cite{BBCK16} Lemma 2.1 for the fact that $\phi_-$ defines indeed a Laplace exponent of a L\'evy process.} of $\eta^-$ is given by
	\begin{align}\label{Equation Laplace exponent phi of eta^-}
		\phi_-(q)=\kappa(\omega_-+q),\quad q\in\mathbb{R}.
	\end{align}
	Since $\phi_-'(0)<0$ by \eqref{Equation Cramer hypothesis}, we know that $Y^-$ is absorbed at 0 in finite time almost surely. By definition of Lamperti's time-change, its absorption time $I$ is given by
	\begin{align}\label{Equation definition I}
		I=\int_0^\infty\exp\left(-\alpha\eta^-(t)\right)\mathrm{d}t.
	\end{align}
	Thus written, $I$ is known as an \textit{exponential functional} of the L\'evy process $\eta^-$, and this kind of random variables has been extensively studied (see e.g. \cite{PS16} and references therein).
	
	One of the interests of this spinal decomposition is that it allows us to study $A$, e.g. through the forthcoming lemma \ref{Lemma expect A(epsilon) = distr function of I} that states an explicit and fruitful connection between $I$ and $A$.
	
	\paragraph{Conditioning the spine to grow indefinitely.}
	A second spinal decomposition was also introduced in \cite{BBCK16}.
	Under assumption \eqref{Equation Cramer hypothesis}, the process $n\mapsto\sum_{|u|=n}\chi_u(0)^{\omega_+}$ indexed by generation is also a martingale, but now with terminal value $0$ almost surely.
	One can define the joint law $\mathcal{P}_x^+$ of $(\chi_u:u\in\mathbb{U})$ starting from an Eve cell of initial size $x>0$ and a leaf $\widehat{\ell}$, such that if $\Gamma_n$ is an event measurable with respect to the sigma-field generated by the cells at generations at most $n\geq 0$, then
	\begin{align*}
		\mathcal{P}_x^+(\Gamma_n)
		=x^{-\omega_+}\mathcal{E}_x\left(\mathds{1}_{\Gamma_n}\sum_{|u|=n+1}\chi_u(0)^{\omega_+}\right).
	\end{align*}
	In \cite{BBCK16}, the ancestors of $\widehat{\ell}$ are selected as follows: let $\widehat{\ell}(n+1)\in\mathbb{N}^{n+1}$ be the parent of $\widehat{\ell}$ at generation $n+1$, then its conditional law is
	\begin{align*}
		\mathcal{P}_x^+\left(\left.\widehat{\ell}(n+1)=v\right|(\chi_u)_{|u|\leq n}\right)
		=\frac{\chi_v(0)^{\omega_+}}{\sum_{|u|=n+1}\chi_u(0)^{\omega_+}},\qquad\forall v\in\mathbb{N}^{n+1}.
	\end{align*}
	The law of the spine is again that of a positive self-similar Markov process $(Y^+(t))_{t\geq 0}=(\exp(\eta^+(\tau_t)))_{t\geq 0}$ with index $\alpha$, that we denote by $\mathbb{P}_x^+$ when $Y^+$ starts from $x>0$.
	Assigning the role of the Eve cell to the spine instead of $\chi_\emptyset$ reorders the genealogy of the cell system, but leaves the law of $\mathbf{X}$ unchanged, see Section 4.2 in \cite{BBCK16}.
	We shall henceforth write $\mathcal{P}_x^+$ for the distribution of the reordered cell system where $\chi_\emptyset$ has law $\mathbb{P}_x^+$, and its daughters generate independent growth-fragmentations given the spine, with laws $(\mathcal{P}_y)_{y>0}$ according to their initial sizes.
	
	The L\'evy measure $\Pi^+$ of the L\'evy process $\eta^+$ is given by
	\begin{align}\label{Equation Pi Levy measure of eta^+}
		\Pi^+(\mathrm{d}y)=e^{\omega_+y}(\Lambda+\widetilde{\Lambda})(\mathrm{d}y),
	\end{align}
	and the Laplace exponent of $\eta^+$ is
	\begin{align}\label{Equation Laplace exponent phi of eta^+}
		\phi_+(q)=\kappa(\omega_++q),\quad q\geq 0.
	\end{align}
	Since $\kappa'(\omega_+)>0$, the process $\eta^+$ diverges to $\infty$ almost surely, which is therefore also the case of $Y^+$ (but does not explode in finite time since $\alpha<0$).
	
	It is possible to make sense of $Y^+$ starting from 0 as the limit of $\mathbb{P}_x^+$ as $x\to 0+$.
	We can thus define $\mathcal{P}_0^+$, as the law of the growth-fragmentation whose Eve cell has law $\mathbb{P}_0^+$, see Corollary 4.4 in \cite{BBCK16}.
	Note that in this latter result, the convergence of $\mathcal{P}_x^+$ towards $\mathcal{P}_0^+$ as $x\to 0$ is proved only in the sense of finite dimensional distributions. This does not allow us, for instance, to study $A$ under $\mathcal{P}_0^+$ as the limit of $A$ under $\mathcal{P}_x^+$
	
	To emphasize with which process we are working, we shall write $\mathbb{P}_x^+$ (respectively $\mathbb{P}_x^-$) for the law of $Y^+$ (respectively $Y^-$) started from $x>0$, and $\mathbb{E}_x^+$ (respectively $\mathbb{E}_x^-$) for the induced expectation operator.

	\section{Main results}\label{Sectionmain results}
	
	Before stating the results of this paper, we make a last assumption on the parameters.
	Recall that $\Lambda$ denotes the L\'evy measure of $\xi$.
	We shall assume that its left-tail $\overline{\Lambda}:x\mapsto\Lambda(\intervalleoo{-\infty}{-x})$ is regularly varying at $0+$ with index $-\rho$, for some $\rho>0$ that satisfies
		\begin{align}\label{Equation assumption rho}
			\max(2\omega_--\omega_+,-\alpha)<\rho<\omega_-.
		\end{align}
	This is equivalent to regular variation with index $\omega_--\rho$ of the tail $x\mapsto\Pi^-(\intervalleoo{-\infty}{\log(x)})$ as $x\to 0+$.
	Indeed, suppose \eqref{Equation assumption rho} holds and recall that by \eqref{Equation Pi Levy measure of eta^-}, we have that
	\begin{align*}
			\Pi^-(\intervalleoo{-\infty}{\log(x)})
			&=\int_{\intervalleoo{-\infty}{\log(x)}}e^{\omega_- y}\Lambda(\mathrm{d}y)+\int_{\intervalleoo{\log(1-x)}{0}}(1-e^{y})^{\omega_-}\Lambda(\mathrm{d}y).
	\end{align*}
	We see that
	\begin{align*}
		\int_{\intervalleoo{-\infty}{\log(x)}}e^{\omega_- y}\Lambda(\mathrm{d}y)
		\leq x^{\omega_-}\overline{\Lambda}(\log(x))
		=o(x^{\omega_-}). 
	\end{align*}
	To estimate the second integral, we use \eqref{Equation assumption rho} and Theorem 1.6.5 in \cite{BGT87}, which then gives us that
	\begin{align}\label{Equation tail Pi^-}
		\Pi^-(\intervalleoo{-\infty}{\log(x)})
		&\underset{x\to 0+}{\sim}\int_{\intervalleoo{\log(1-x)}{0}}(1-e^{y})^{\omega_-}\Lambda(\mathrm{d}y)
		\underset{x\to 0+}{\sim}\overline{\Lambda}\left(x\right)x^{\omega_-}\frac{\rho}{\omega_--\rho}.
	\end{align}
	The converse also follows from the same theorem, that is if $x\mapsto\Pi^-(\intervalleoo{-\infty}{\log(x)})$ has regular variation at $0+$ with index $\omega_--\rho$ then it is also the case for $\overline{\Lambda}$ with index $-\rho$.
	Indeed, coming back to the first part of \eqref{Equation tail Pi^-} and using $1-e^{-y}\sim y$ and $\log(1-y)\sim y$ as $y\to 0+$, we see that
	\begin{align*}
		\Pi^-(\intervalleoo{-\infty}{\log(x)})
		&\underset{x\to 0+}{\sim}\int_{\intervalleoo{-x}{0}}|y|^{\omega_-}\Lambda(\mathrm{d}y).
	\end{align*}
	This is enough to conclude using \cite{BGT87} Theorem 1.6.5 (see the discussion following the proof of the theorem).
	
	Similarly, one can show that regular variation with index $-\rho$ of $\overline{\Lambda}$ at $0+$ is equivalent to regular variation with index $\omega_+-\rho$ of $x\mapsto\Pi^+(\intervalleoo{-\infty}{\log(x)})$ as $x\to 0+$.
	
	Recall the definition of $I$ from \eqref{Equation definition I}.
	
	\begin{thm}\label{Theorem area GF}
		Consider a growth-fragmentation $\mathbf{X}$ such that \eqref{Equation Cramer hypothesis} and \eqref{Equation assumption rho} are satisfied.
		Then, for every $x>0$, it holds that
		\begin{align*}
			\frac{\epsilon^{-(1+\frac{\omega_-}{|\alpha|})}}{\overline{\Lambda}(\epsilon^{1/|\alpha|})}A(\epsilon)
			\underset{\epsilon\to 0}{\longrightarrow}\frac{|\alpha|\rho}{(\omega_--\rho)(\omega_-+|\alpha|-\rho)}\mathbb{E}_1^-(I^{\frac{\omega_--\rho}{\alpha}})x^{\alpha+\rho},
		\end{align*}
		$\mathcal{P}_x$-almost surely and in $\mathbb{L}^1$, where the expectation is finite.
	\end{thm}
	
	Since $\alpha<0$, we know from \cite{BBCK16} Proposition 4.1 that for any $x>0$, $\mathcal{P}_x^+$ is absolutely continuous with respect to the restriction of $\mathcal{P}_x$ to the natural filtration of the growth-fragmentation and the spine up to time $t>0$.
	Therefore the above convergence also holds $\mathcal{P}_x^+$-a.s.

	In Section 6, we shall see that the following interesting result easily follows from self-similarity.
	\begin{prop}\label{Theorem stationary area from 0}
		Suppose that \eqref{Equation Cramer hypothesis} holds.
		Then under $\mathcal{P}_0^+$, the process
		\begin{align*}
			\left(e^{\frac{\omega_-}{\alpha}u}A\left(e^u\right)\right)_{u\in\mathbb{R}}
		\end{align*}
		is stationary.
		In particular, the law of $t^{\frac{\omega_-}{\alpha}}A(t)$ does not depend on $t>0$.
	\end{prop}
	
	The next two propositions provide some bound about the almost sure behaviour of this area and shows that almost surely, $A(t)$ does not deviate from $t^{\omega_-/|\alpha|}$ by more than a power of $|\log(t)|$:
	
	\begin{prop}\label{Theorem upper envelope A(t)}
		Suppose that \eqref{Equation Cramer hypothesis} and $\kappa(\omega_++\omega_-+\alpha)<\infty$ hold.
		For all $\delta>0$, we have that
		\begin{align*}
			\limsup_{t\to 0+\text{ or }\infty}|\log(t)|^{-1-\delta}t^{\frac{\omega_-}{\alpha}}A(t)=0,\quad\mathcal{P}_0^+\text{-a.s.}
		\end{align*}
	\end{prop}
	
	\begin{prop}\label{Proposition lower envelope A(t)}
		Suppose that \eqref{Equation Cramer hypothesis} and \eqref{Equation assumption rho} hold.
		Then, there exists $q_0>0$ such that for all $q>q_0$, we have that
		\begin{align*}
			\liminf_{t\to 0+\text{ or }\infty}|\log(t)|^{q}t^{\frac{\omega_-}{\alpha}}A(t)=\infty,\quad\mathcal{P}_0^+\text{-a.s.}
		\end{align*}
	\end{prop}
	
	The proof of Proposition \ref{Proposition lower envelope A(t)} provides an explicit $q_0$, with a rather complicated expression given in Remark \ref{Remark exponent lower envelope} in Section 5.
	We do not claim however that the bounds of the two above Propositions are optimal. 
	
	In Section \ref{Section Application to random maps}, we shall recall the connection between growth-fragmentations and a family of random surfaces, and apply the above results to the latter, refining in particular some results for the Brownian map.
	
	\section{Area near the origin}\label{Section area near the origin}
	We assume in this section that \eqref{Equation Cramer hypothesis} and \eqref{Equation assumption rho} hold.
	Our goal is to prove Theorem \ref{Theorem area GF}.
	We first look at the expectation of $A(\epsilon)$ as $\epsilon\to 0+$.
	\subsection{Behaviour of the expectation}

	Under $\mathcal{P}_1^-$, the spine is chosen according to $\mathcal{A}$. We can hence study $A$ through the lifetime of the spine $I$, thanks to the following relation.
	
	\begin{lem}\label{Lemma expect A(epsilon) = distr function of I}
		For all $t\geq 0$, it holds that
		\begin{align*}
			\mathcal{E}_1(A(t))=\mathbb{P}_1^-(I\leq t).
		\end{align*}
	\end{lem}
	This is a straightforward consequence of Lemmas 2.1 and 2.2 in \cite{G18}.
	Thanks to Lemma \ref{Lemma expect A(epsilon) = distr function of I}, we can deduce
	  
	\begin{lem}\label{Lemma asymp expect A(espilon)}
		It holds that that
		\begin{align*}
			\mathcal{E}_1(A(\epsilon))\underset{\epsilon\to 0}{\sim}\frac{|\alpha|\rho}{(\omega_--\rho)(\omega_-+|\alpha|-\rho)}\mathbb{E}_1^-(I^{\frac{\omega_--\rho}{\alpha}})\epsilon^{1+\frac{\omega_-}{|\alpha|}}\overline{\Lambda}(\epsilon^{1/|\alpha|}),
		\end{align*}
		where the expectation on the right-hand side is finite.
	\end{lem}
	
	\begin{proof}
		Let $\Pi_\alpha^-$ be the L\'evy measure of $|\alpha|\eta^-$.
		The behaviour of $\mathbb{P}_1^-(I\leq t)$ as $t\to 0$ is given in Theorem 7 of \cite{AR15}.
		More precisely, provided that $\Pi_\alpha^-(\intervalleoo{-\infty}{-(x+y)})/\Pi_\alpha^-(\intervalleoo{-\infty}{-x})\sim e^{-\gamma y}$ as $x\to\infty$ with $\mathbb{E}_1^-(\exp(|\alpha|\gamma\eta^-(1)))<1$ (corresponding to $|\alpha|\gamma\in\intervalleoo{0}{\omega_+-\omega_-}$ by \eqref{Equation Laplace exponent phi of eta^-}), it holds that 
		\begin{align*}
			\mathbb{P}_1^-(I\leq t)
			\underset{t\to 0}{\sim}\frac{\mathbb{E}_1^-(I^{-\gamma})}{1+\gamma}t\Pi_\alpha^-(\intervalleoo{-\infty}{\log(1/t)}).
		\end{align*}
		By \eqref{Equation tail Pi^-}, we have that
		\begin{align*}
			\Pi_\alpha^-(\intervalleoo{-\infty}{-x})
			&\underset{x\to \infty}{\sim}\overline{\Lambda}\left(e^{x/\alpha}\right)e^{\omega_-x/\alpha}\frac{\rho}{\omega_--\rho}.
		\end{align*}
		We thus see that $\Pi_\alpha(\intervalleoo{-\infty}{-x-y})/\Pi_\alpha(\intervalleoo{-\infty}{-x})\to \exp(-\frac{\omega_--\rho}{|\alpha|} y)$ as $x\to\infty$.
		It remains to check that $\omega_--\rho<\omega_+-\omega_-$, which is true since $2\omega_--\omega_+<\rho<\omega_-$ by \eqref{Equation assumption rho}.
		We therefore get by \cite{AR15} Theorem 7 that 
		\begin{align*}
			\mathbb{P}_1^-(I\leq \epsilon)
			&\underset{\epsilon\to 0+}{\sim}\left(1+\frac{\omega_--\rho}{|\alpha|}\right)^{-1}\mathbb{E}_1^-\left(I^{-\frac{\omega_--\rho}{|\alpha|}}\right)\epsilon\Pi_\alpha^-(\intervalleoo{-\infty}{\log(\epsilon)})\\
			&\underset{\epsilon\to 0+}{\sim}\frac{|\alpha|\rho}{(\omega_--\rho)(\omega_-+|\alpha|-\rho)}\mathbb{E}_1^-\left(I^{-\frac{\omega_--\rho}{|\alpha|}}\right)\epsilon^{1+\frac{\omega_-}{|\alpha|}}\overline{\Lambda}(\epsilon^{1/|\alpha|}),
		\end{align*}
		as claimed, where the expectation is finite.
	\end{proof}
	
	We shall need a bit more than the speed of convergence to 0 of the first moment.
	
	\begin{lem}\label{Lemma speed A^p}
		For $p=p(\epsilon):=1+1/\sqrt{|\log(\epsilon)|}$, for all $0<q<1+\frac{\omega_--\rho}{|\alpha|}$ it holds that
		\begin{align*}
			\mathcal{E}_1\left(A(\epsilon)^p\right)=o\left(\epsilon^q\right),\quad\text{as}\quad\epsilon\to 0.
		\end{align*}
	\end{lem}
	
	\begin{proof}
		Let $n\geq 1$ be arbitrary large. We write
		\begin{align*}
			\mathcal{E}_1\left(A(\epsilon)^p\right)
			&\leq\mathcal{E}_1\left(A(\epsilon)\mathds{1}_{\left\{A(\epsilon)\leq \epsilon^{-n}\right\}}\right)\times \epsilon^{-n(p-1)}+\mathcal{E}_1\left(A(\epsilon)^p\mathds{1}_{\left\{A(\epsilon)>\epsilon^{-n}\right\}}\right).
		\end{align*}
		By Lemma \ref{Lemma asymp expect A(espilon)}, the first term is of order $\epsilon^{1+\frac{\omega_-}{-\alpha}-n(p-1)}\overline{\Lambda}(\epsilon^{1/|\alpha|})$, with $p-1=1/\sqrt{-\log(\epsilon)}$ going to 0 with $\epsilon$. It remains to bound the second term. Take $q,q'>1$ such that $1/q+1/q'=1$ and $pq<\omega_+/\omega_-$, H\"older's Inequality then yields that
		\begin{align*}
			\mathcal{E}_1\left(A(\epsilon)^p\mathds{1}_{\left\{A(\epsilon)>\epsilon^{-n}\right\}}\right)
			&\leq\mathcal{E}_1\left(A(\epsilon)^{pq}\right)^{1/q}\times\mathcal{P}(A(\epsilon)>\epsilon^{-n})^{1/q'}.
		\end{align*}
		Since $pq<\omega_+/\omega_-$ we know that $\mathcal{E}_1(A(\epsilon)^{pq})<\infty$ by Lemma 2.3 in \cite{BBCK16}. Markov's Inequality shows that the probability on the right-hand side is bounded from above by
		\begin{align*}
			\mathcal{E}_1(A(\epsilon))^{1/q'}\times\epsilon^{n/q'}.
		\end{align*}
		Since $n$ was chosen arbitrary large, this concludes the proof.
	\end{proof}
	
	The proof above would work for any $p(\epsilon)$ converging to 1 from above as $\epsilon\to 0$, but will shall use it with this specific choice of $p(\epsilon)$ in order to have a technical result, that we record now for later use.
	
	\begin{lem}\label{Lemma equiv p(epsilon)}
		For $p=p(\epsilon)$ as in Lemma \ref{Lemma speed A^p}, it holds that
		\begin{align*}
			\epsilon^{1+\frac{p\omega_-}{|\alpha|}}\overline{\Lambda}(\epsilon^{1/|\alpha|})=o\left(\epsilon^{p(1+\frac{\omega_-}{|\alpha|})}\overline{\Lambda}(\epsilon^{1/|\alpha|})^{p}\right),\quad\text{as }\epsilon\to 0+.
		\end{align*}
	\end{lem}
	
	\begin{proof}
		By $\eqref{Equation assumption rho}$ and Theorem 1.4.1(iii) in \cite{BGT87}, we know that there exists a function $\ell:\intervalleoo{0}{\infty}\to\intervalleoo{0}{\infty}$ with slow variation at $0$ such that $\overline{\Lambda}(\epsilon)\sim\epsilon^{-\rho}\ell(\epsilon)$.
		We write
		\begin{align*}
			&\log\left(\frac{\epsilon^{1+\frac{p\omega_-}{|\alpha|}}\overline{\Lambda}(\epsilon^{1/|\alpha|})}{\epsilon^{p(1+\frac{\omega_-}{|\alpha|})}\overline{\Lambda}(\epsilon^{1/|\alpha|})^{p}}\right)
			=(1-p)\log(\epsilon)+(1-p)\log\left(\overline{\Lambda}(\epsilon^{1/|\alpha|})\right)\\
			&\hspace{3.8cm}\underset{\epsilon\to 0+}{\sim}(1-p)\log(\epsilon)+\frac{(1-p)\rho}{\alpha}\log(\epsilon)+(1-p)\log(\ell(\epsilon^{1/|\alpha|})).
		\end{align*}
		Since $\ell$ has slow variation, the terms with $\log(\epsilon)$ dominates.
		By definition of $p$ and since $|\alpha|<\rho$ by \eqref{Equation assumption rho}, we see that $(1-p)(1+\rho/\alpha)\log(\epsilon)=(1+\rho/\alpha)\sqrt{|\log(\epsilon)|}\to -\infty$ as $\epsilon\to 0+$, which entails the claim.
	\end{proof}

	\subsection{The almost sure behaviour}
	
	Throughout this subsection, we shall only work under $\mathcal{P}_1$, since the other cases follow from the self-similarity.
	Indeed, under $\mathcal{P}_x$, we have $(A(t))_{t\geq 0}\stackrel{d}{=}(x^{\omega_-}A(t x^\alpha))_{t\geq 0}$ under $\mathcal{P}_1$.
	It means that if Theorem \ref{Theorem area GF} is true under $\mathcal{P}_1$, we easily deduce that under $\mathcal{P}_x$,
	\begin{align*}
		\frac{\epsilon^{-(1+\frac{\omega_-}{|\alpha|})}}{\overline{\Lambda}\left(\epsilon^{1/|\alpha|}\right)}A(\epsilon)
		&\stackrel{d}{=}x^{\alpha}\frac{\overline{\Lambda}\left((\epsilon x^\alpha)^{1/|\alpha|}\right)}{\overline{\Lambda}\left(\epsilon^{1/|\alpha|}\right)}\cdot\frac{(\epsilon x^{\alpha})^{-(1+\frac{\omega_-}{|\alpha|})}}{\overline{\Lambda}\left((\epsilon x^\alpha)^{1/|\alpha|}\right)}A(\epsilon x^{\alpha}),\quad\text{under }\mathcal{P}_1\\
		&\!\!\!\!\underset{\epsilon\to 0+}{\sim}x^{\rho+\alpha}\lim_{\epsilon\to 0+}\frac{(\epsilon x^{\alpha})^{-(1+\frac{\omega_-}{|\alpha|})}}{\overline{\Lambda}\left((\epsilon x^\alpha)^{1/|\alpha|}\right)}A(\epsilon x^{\alpha}),\hspace{1.3cm}\text{under }\mathcal{P}_1,
	\end{align*}
	where we used the assumption of regular variation of $\overline{\Lambda}$ at $0+$, as written in \eqref{Equation assumption rho}.
	
	Recall that $\xi$ is the L\'evy process associated with $X$ by Lamperti's transformation, and that under $\mathcal{P}_1$, $\chi_\emptyset$ is distributed as $X$.
	Let $\sigma:s\mapsto\int_0^s e^{-\alpha\xi(u)}\mathrm{d}u$ be the inverse change time of $(\tau_t)_{t\geq 0}$.
	Define the compensated process $(M_t)_{t\geq 0}$ as
	\begin{align*}
		M_t
		:=&\sum_{s\leq t}(\Delta_- e^{\xi(s)})^{\omega_-}A_s\left((\sigma_t-\sigma_s)(\Delta_-e^{\xi(s)})^\alpha\right)-S_t,\\
		\shortintertext{where the $A_s$'s are i.i.d. copies of $A$ under $\mathcal{P}_{1}$, independent from $\xi$, with}
		S_t:=&\int_0^t\mathrm{d}se^{\omega_-\xi(s)}\int_{\intervalleoo{-\infty}{0}}\Lambda(\mathrm{d}y)(1-e^y)^{\omega_-}\mathbb{P}_1^-\left(I\leq (1-e^y)^{\alpha}e^{\alpha\xi(s)}(\sigma_t-\sigma_s)|\mathcal{F}_s\right),
	\end{align*}
	where $I$ is independent from $\xi$ and $(\mathcal{F}_s)_{s\geq 0}$ is the natural filtration of $\xi$.
	
	\begin{lem}\label{Lemma M martingale and M(tau)=A-S}
		The process $(M_t)_{t\geq 0}$ is a martingale.
		Moreover, identifying $\chi_{\emptyset}$ with $X=\exp(\xi_\tau)$, we have that $M_{\tau_t}=A(t)-S_{\tau_t}$ for all $t\geq 0$.
	\end{lem}
	\begin{proof}
		Thanks to Lemma \ref{Lemma expect A(epsilon) = distr function of I} and independence of the $A_s$'s with $\xi$, one sees that $(S_t)_{t\geq 0}$ is the predictable compensator of the series in $(M_t)_{t\geq 0}$, so that the latter is a martingale.
		We then write
		\begin{align*}
			M_{\tau_{t}}
			&=\sum_{s\leq \tau_t}(\Delta_- e^{\xi(s)})^{\omega_-}A_s\left((t-\sigma_s)(\Delta_-e^{\xi(s)})^\alpha\right)-S_{\tau_t}\\
			&=\sum_{s\leq t}(\Delta_- e^{\xi(\tau_s)})^{\omega_-}A_{\tau_s}\left((t-s)(\Delta_-e^{\xi(\tau_s)})^\alpha\right)-S_{\tau_t}\\
			&=A(t)-S_{\tau_t},
		\end{align*}
		by Lemma \ref{Lemma Markov-branching property of A} (where we reindexed the $A_s$'s), as claimed. 
	\end{proof}
	
	The main idea of the current section is to use Lemma \ref{Lemma M martingale and M(tau)=A-S}, namely the martingale property of $M$ and its relation to $A$, to show that $A(\epsilon)\sim S_{\tau_\epsilon}$.
	In this direction, we first establish the following:
	
	\begin{lem}\label{Lemma predictable compensator as E(A(epsilon))}
		$\mathcal{P}_1$-almost surely, it holds that
		\begin{align*}
			S_{\tau_{\epsilon}}&\underset{\epsilon\to 0}{\sim}\frac{|\alpha|\rho}{(\omega_--\rho)(\omega_-+|\alpha|-\rho)}\mathbb{E}_1^-(I^{-\frac{\omega_--\rho}{|\alpha|}})\overline{\Lambda}\left(\epsilon^{1/|\alpha|}\right)\epsilon^{1+\frac{\omega_-}{|\alpha|}}.
		\end{align*}
	\end{lem}
	
	\begin{proof}
		Recall \eqref{Equation Lamperti tau_t}. After a change of variables and replacing $\exp(\xi(\tau_s))$ by $X(s)$, we have that
		\begin{align*}
			S_{\tau_\epsilon}
			&=\int_0^\epsilon\mathrm{d}sX(s)^{\omega_-+\alpha}\int_{\intervalleoo{-\infty}{0}}\Lambda(\mathrm{d}y)(1-e^y)^{\omega_-}\mathbb{P}_1^-\left(I\leq (1-e^y)^{\alpha}X(s)^{\alpha}(\epsilon-s)|\mathcal{F}_{\tau_s}\right)
		\end{align*}
		It is well known that the law of $I$ is absolutely continuous, see e.g. \cite{BLM08} Theorem 3.9.
		Hence, since $X(s)\sim 1$ as $s\to 0+$, one sees that
		\begin{align*}
			S_{\tau_\epsilon}
			&\underset{\epsilon\to 0+}{\sim}\int_0^\epsilon\mathrm{d}s\int_{\intervalleoo{-\infty}{0}}\Lambda(\mathrm{d}y)(1-e^y)^{\omega_-}\mathbb{P}_1^-\left(I\leq(\epsilon-s)(1-e^y)^\alpha\right),\quad \mathcal{P}_1\text{-almost surely.}
		\end{align*}
		We express the right-hand side in the form
\begin{align*}
	&\int_{\intervalleoo{-\infty}{0}}\Lambda(\mathrm{d}y)(1-e^y)^{\omega_-}\int_0^\epsilon\mathrm{d}s\mathbb{P}_1^-(I\leq s(1-e^y)^{\alpha})\\
	&\hspace{2cm}=\int_{\intervalleoo{-\infty}{0}}\Lambda(\mathrm{d}y)(1-e^y)^{\omega_-+|\alpha|}\int_0^{\epsilon(1-e^y)^\alpha}\mathrm{d}s\mathbb{P}_1^-(I\leq s)\\
	&\hspace{2cm}=C_1(\epsilon)+C_2(\epsilon),
\end{align*}
where $C_1(\epsilon)$ is the part of the first integral with domain restricted to $E_1:=\intervalleoo{-\infty}{\log\left(1-\epsilon^{\frac{1-\delta}{|\alpha|}}\right)}$ for $\delta>0$ arbitrary small, and $C_2(\epsilon)$ is that restricted to $E_2:=\intervallefo{\log\left(1-\epsilon^{\frac{1-\delta}{|\alpha|}}\right)}{0}$.
We first address $C_1(\epsilon)$.
Note that for $y\in E_1$, we have that $s\leq\epsilon (1-e^y)^\alpha\leq \epsilon^\delta$.
Appealing to lemmas \ref{Lemma expect A(epsilon) = distr function of I} and \ref{Lemma asymp expect A(espilon)}, there exists a constant $C$ such that
\begin{align*}
	\mathbb{P}_1^-(I\leq s)\leq Cs^{1+\frac{\omega_-}{|\alpha|}}\overline{\Lambda}(s^{1/|\alpha|}),\quad\forall s\leq\epsilon^\delta.
\end{align*}
In the following, we shall keep writing $C$ for any positive and finite constant that does not depend on $\epsilon$ and that may change from line to line.
We have
\begin{align*}
	C_1(\epsilon)
	&\leq C\int_{E_1}\Lambda(\mathrm{d}y)(1-e^y)^{\omega_-+|\alpha|}\int_0^{\epsilon^{\delta}} s^{1+\frac{\omega_-}{|\alpha|}}\overline{\Lambda}(s^{1/|\alpha|})\mathrm{d}s\\
	&\leq C\int_{E_1}\Lambda(\mathrm{d}y)(1-e^y)^{\omega_-+|\alpha|}\times\epsilon^{\delta(2+\frac{\omega_-}{|\alpha|})}\overline{\Lambda}(\epsilon^{\delta/|\alpha|}),
\end{align*}
where we used Theorem 1.5.11(ii) of \cite{BGT87}.
Thanks to Theorem 1.6.4 of the same book, we conclude that
\begin{align*}
	C_1(\epsilon)&=O\left(\epsilon^{(1-\delta)(1+\frac{\omega_-}{|\alpha|})}\overline{\Lambda}(\epsilon^{(1-\delta)/|\alpha|})\times\epsilon^{\delta(2+\frac{\omega_-}{|\alpha|})}\overline{\Lambda}(\epsilon^{\delta/|\alpha|})\right)\\
	&=O\left(\epsilon^{\delta}\times\epsilon^{1+\frac{\omega_-}{|\alpha|}}\overline{\Lambda}(\epsilon^{(1-\delta)/|\alpha|})\overline{\Lambda}(\epsilon^{\delta/|\alpha|})\right)\\
	&=o\left(\epsilon^{1+\frac{\omega_-}{|\alpha|}}\overline{\Lambda}(\epsilon^{1/|\alpha|})\right),
\end{align*}
where the last line follows straightforwardly from the existence of a slowly varying function $\ell:\intervalleoo{0}{\infty}\to\intervalleoo{0}{\infty}$ at the origin such that $\overline{\Lambda}(x)=x^{-\rho}\ell(x)$, see \cite{BGT87} Theorem 1.4.1(iii).

We now investigate $C_2(\epsilon)$.
Let $k$ denote the density of $I$ under $\mathbb{P}_1^-$.
After applying Tonelli's Theorem, we have that
\begin{align}\label{Equation C_2}
	C_2(\epsilon)
	&=\int_0^{\epsilon^{\delta}}k(x)\mathrm{d}x\int_{E_2}\Lambda(\mathrm{d}y)(1-e^y)^{\omega_-}(\epsilon-x(1-e^y)^{|\alpha|})+\int_{\epsilon^{\delta}}^{\infty} k(x)xf(\epsilon/x)\mathrm{d}x,
\end{align}
where
\begin{align*}
	f(u):=\int_{\intervalleoo{\log\left(1-u^{\frac{1}{|\alpha|}}\right)}{0}}\Lambda(\mathrm{d}y)(1-e^y)^{\omega_-}(u-(1-e^y)^{|\alpha|}).
\end{align*}
Theorem 1.6.5 of \cite{BGT87} allows us to bound the first term in the right-hand side of \eqref{Equation C_2} by
\begin{align*}
	\mathbb{P}_1^-\left(I\leq\epsilon^{\delta}\right)\times\epsilon\int_{E_2}\Lambda(\mathrm{d}y)(1-e^y)^{\omega_-}
	&\underset{\epsilon\to 0+}{\sim}C\mathbb{P}_1^-\left(I\leq\epsilon^\delta\right)\epsilon^{-\delta\frac{\omega_--\rho}{|\alpha|}}\epsilon^{1+\frac{\omega_-}{|\alpha|}}\overline{\Lambda}(\epsilon^{(1-\delta)/|\alpha|})\\
	&\ \ =o\left(\epsilon^{1+\frac{\omega_-}{|\alpha|}}\overline{\Lambda}(\epsilon^{1/|\alpha|})\right),
\end{align*}
with the same argument as for $C_1(\epsilon)$.

We turn our attention to the second term in \eqref{Equation C_2}.
First, with the help of Theorem 1.6.5 of \cite{BGT87}, it can be shown that
\begin{align*}
	f(u)\underset{u\to 0+}{\sim}&\left(\frac{\rho}{\omega_--\rho}-\frac{\rho}{\omega_-+|\alpha|-\rho}\right)u^{1+\frac{\omega_-}{|\alpha|}}\overline{\Lambda}(u^{1/|\alpha|})\\
	=\ &\frac{|\alpha|\rho}{(\omega_--\rho)(\omega_-+|\alpha|-\rho)}u^{1+\frac{\omega_-}{|\alpha|}}\overline{\Lambda}(u^{1/|\alpha|})
\end{align*}
Note that $\epsilon/x\leq \epsilon^{1-\delta}$ for all $x>\epsilon^\delta$.
The estimate above yields
\begin{align*}
	\int_{\epsilon^\delta}^{\infty} k(x)xf(\epsilon/x)\mathrm{d}x
	\underset{\epsilon\to 0}{\sim}\frac{|\alpha|\rho}{(\omega_--\rho)(\omega_-+|\alpha|-\rho)}\epsilon^{1+\frac{\omega_-}{|\alpha|}}\int_{\epsilon^\delta}^{\infty}k(x)x^{\frac{\omega_-}{\alpha}}\overline{\Lambda}((\epsilon/x)^{1/|\alpha|})\mathrm{d}x.
\end{align*}
By dominated convergence, we then see that
\begin{align*}
	\left(\epsilon^{1+\frac{\omega_-}{|\alpha|}}\overline{\Lambda}(\epsilon^{1/|\alpha|})\right)^{-1}\int_{\epsilon^\delta}^{\infty} k(x)xf(\epsilon/x)\mathrm{d}x
	\underset{\epsilon\to 0}{\longrightarrow}\frac{|\alpha|\rho}{(\omega_--\rho)(\omega_-+|\alpha|-\rho)}\mathbb{E}_1^-\left(I^{\frac{\omega_--\rho}{\alpha}}\right),
\end{align*}
as claimed.
\end{proof}
	
	We are ready to provide the proof of Theorem \ref{Theorem area GF}. 
	
	\begin{proof}[Proof of Theorem \ref{Theorem area GF}]
	Provided that $\epsilon\mapsto M_\epsilon$ has regular variation at 0, one has that $M_{\tau_\epsilon}\sim M_\epsilon$ as $\epsilon\to 0+$, as a consequence of $\tau_\epsilon\sim\epsilon$.
	Thus, thanks to Lemmas \ref{Lemma M martingale and M(tau)=A-S} and \ref{Lemma predictable compensator as E(A(epsilon))}, it suffices to show that 
	\begin{align}\label{Equation to show}
		\lim_{\epsilon\to 0}\left(\epsilon^{1+\frac{\omega_-}{|\alpha|}}\overline{\Lambda}(\epsilon^{1/|\alpha|})\right)^{-1}M_\epsilon=0,\quad\mathcal{P}_1\text{-a.s.}
	\end{align}
	Define the stopping time 
	\begin{align*}
		T:=\inf\left\{s>0:|\xi(s)|>1\right\}.
	\end{align*}
	We proceed as follows: we know that $\mathcal{M}$ has finite moment of order $p$ as soon as $p<\omega_+/\omega_-$ by \cite{BBCK16} Lemma 2.3. Let $p=p(\epsilon)=1+1/\sqrt{|\log(\epsilon)|}$ as in Lemma \ref{Lemma speed A^p} and let $a>0$ be arbitrary small, we have by Markov's Inequality that 
	\begin{align*}
		\mathcal{P}_1\left(\sup_{s\leq\epsilon\wedge T}\frac{|M_s|}{\epsilon^{1+\frac{\omega_-}{|\alpha|}}\overline{\Lambda}(\epsilon^{1/|\alpha|})}\geq a\right)
		&\leq \left(a\epsilon^{1+\frac{\omega_-}{|\alpha|}}\overline{\Lambda}(\epsilon^{1/|\alpha|})\right)^{-p}\mathcal{E}_1\left(\sup_{s\leq\epsilon\wedge T}|M_s|^{p}\right)\\
		&\leq (6p)^p\left(a\epsilon^{1+\frac{\omega_-}{|\alpha|}}\overline{\Lambda}(\epsilon^{1/|\alpha|})\right)^{-p}\mathcal{E}_1\left([M]_{\epsilon\wedge T}^{p/2}\right),
	\end{align*}
	thanks to Burkholder-Davis-Gundy Inequality, where $[M]$ denotes the quadratic variation of $M$ (see Theorem 92 p.304 in \cite{DM2} for the constant $(6p)^p$).
	In particular, since $M$ is a purely discontinuous martingale, appealing to \cite{L76} Section 3(c) we get that the previous quantity is bounded from above by
	\begin{align*}
		(6p)^p\left(a\epsilon^{1+\frac{\omega_-}{|\alpha|}}\overline{\Lambda}(\epsilon^{1/|\alpha|})\right)^{-p}\mathcal{E}_1\left(\sum_{s\leq \epsilon\wedge T}(\Delta e^{\xi(s)})^{p\omega_-}A_s\left((\epsilon\wedge T-s)(\Delta e^{\xi(s)})^{\alpha}\right)^p\right),
	\end{align*}
	Since $T>0$ almost surely by right-continuity of $\xi$, in order to conclude, it is sufficient to show that the expectation above is $o\left(\epsilon^{p(1+\frac{\omega_-}{|\alpha|})}\overline{\Lambda}(\epsilon^{1/|\alpha|})^{p}\right)$, as $\epsilon\to 0$.
	We write
	\begin{align*}
		&\mathcal{E}_1\left(\sum_{s\leq \epsilon\wedge T}(\Delta e^{\xi(s)})^{p\omega_-}A_s\left((\epsilon\wedge T-s)(\Delta e^{\xi(s)})^{\alpha}\right)^p\right)\\
		&\hspace{2cm}=\mathcal{E}_1\left(\sum_{s\leq \epsilon\wedge T}e^{p\omega_-\xi(s-)}(1-e^{\Delta\xi(s)})^{p\omega_-}A_s\left((\epsilon\wedge T-s)e^{\alpha\xi(s-)}(1-e^{\Delta\xi(s)})^{\alpha}\right)^p\right)\\
		&\hspace{2cm}\leq e^{p\omega_-}\mathcal{E}_1\left(\sum_{s\leq \epsilon}(1-e^{\Delta\xi(s)})^{p\omega_-}A_s\left(e^{|\alpha|}(\epsilon-s)(1-e^{\Delta\xi(s)})^{\alpha}\right)^p\right).
	\end{align*}

	(Recall that the $A_s$'s are non-decreasing and non-negative functions.)
	Applying the compensation formula, the upper bound above becomes, after an implicit change of variables,
\begin{align}\label{Equation computation}
	&e^{p\omega_-}e^{|\alpha|}\int_0^{\epsilon}\mathrm{d}s\int_{\intervalleoo{-\infty}{0}}\Lambda(\mathrm{d}y)(1-e^y)^{p\omega_-}\mathcal{E}_1\left(A\left(e^{|\alpha|}s(1-e^y)^\alpha\right)^p\right)\nonumber\\
	&\hspace{1cm}=e^{p\omega_-+|\alpha|}\int_{\intervalleoo{-\infty}{0}}\Lambda(\mathrm{d}y)(1-e^y)^{p\omega_-}\int_0^{\epsilon}\mathrm{d}s\mathcal{E}_1\left(A\left(e^{|\alpha|}s(1-e^y)^\alpha\right)^p\right)\nonumber\\
	&\hspace{1cm}=e^{p\omega_-+|\alpha|}\int_{\intervalleoo{-\infty}{0}}\Lambda(\mathrm{d}y)(1-e^y)^{p\omega_-+|\alpha|}\int_0^{\epsilon(1-e^y)^\alpha}\mathrm{d}s\mathcal{E}_1\left(A\left(e^{|\alpha|}s\right)^p\right)\nonumber\\
	&\hspace{1cm}=e^{p\omega_-+|\alpha|}\int_0^{\epsilon}\mathrm{d}s\mathcal{E}_1\left(A\left(e^{|\alpha|}s\right)^p\right)\int_{\intervalleoo{-\infty}{0}}\Lambda(\mathrm{d}y)(1-e^y)^{p\omega_-+|\alpha|}\nonumber\\
	&\hspace{2cm}+e^{p\omega_-+|\alpha|}\int_{\epsilon}^\infty\mathrm{d}s\mathcal{E}_1\left(A\left(e^{|\alpha|}s\right)^p\right)\int_{\intervalleoo{\log(1-(\epsilon/s)^{\frac{1}{|\alpha|}})}{0}}\Lambda(\mathrm{d}y)(1-e^y)^{p\omega_-+|\alpha|}\nonumber\\
	&\hspace{1cm}\leq e^{p\omega_-+|\alpha|}\int_0^{K\epsilon}\mathrm{d}s\mathcal{E}_1\left(A\left(e^{|\alpha|}s\right)^p\right)\int_{\intervalleoo{-\infty}{0}}\Lambda(\mathrm{d}y)(1-e^y)^{p\omega_-+|\alpha|}\nonumber\\
	&\hspace{2cm}+e^{p\omega_-+|\alpha|}\int_{K\epsilon}^\infty\mathrm{d}s\mathcal{E}_1\left(A\left(e^{|\alpha|}s\right)^p\right)\int_{\intervalleoo{\log(1-(\epsilon/s)^{\frac{1}{|\alpha|}})}{0}}\Lambda(\mathrm{d}y)(1-e^y)^{p\omega_-+|\alpha|},
\end{align}
	where $K$ is a constant intended to be large.
	The first term is negligeable since by Lemma \ref{Lemma speed A^p}, for $\delta>0$ small enough, we have that
	\begin{align*}
		\int_0^{K\epsilon}\mathcal{E}_1\left(A\left(e^{|\alpha|}s\right)^p\right)\mathrm{d}s
		=o(\epsilon^{2+\frac{\omega_--\rho}{|\alpha|}-\delta})
		=o(\epsilon^{p(1+\frac{\omega_--\rho}{|\alpha|})}).
	\end{align*}
	(Recall that $p=1+1/\sqrt{|\log(\epsilon)|}$.)
	To bound the second term, we first note that for all $s>K\epsilon$, we have $\epsilon/s\leq 1/K$, so that Theorem 1.6.5 of \cite{BGT87} entails the existence of a constant $C$ such that
	\begin{align*}
		\int_{\intervalleoo{\log(1-(\epsilon/s)^{\frac{1}{|\alpha|}})}{0}}\Lambda(\mathrm{d}y)(1-e^y)^{p\omega_-+|\alpha|}
		&\leq C\left(\frac{\epsilon}{s}\right)^{\frac{p\omega_-+|\alpha|}{|\alpha|}}\overline{\Lambda}\left((\epsilon/s)^{1/|\alpha|}\right)\\
		&\leq C \epsilon^{1+\frac{p\omega_--\rho}{|\alpha|}}\left(\frac{1}{s}\right)^{1+\frac{p\omega_--\rho}{|\alpha|}}\ell\left((\epsilon/s)^{1/|\alpha|}\right),
	\end{align*}
	where $\ell:\intervalleoo{0}{\infty}\to\intervalleoo{0}{\infty}$ is some slowly varying function at the origin.
	Finally, we get that the second term in the right-hand side of \eqref{Equation computation} is bounded from above by
	\begin{align*}
		\epsilon^{1+\frac{p\omega_--\rho}{|\alpha|}}\int_{0}^\infty\mathrm{d}s\mathcal{E}(A(e^{|\alpha|}s)^p)s^{-1-(p\omega_--\rho)/|\alpha|}\ell\left((\epsilon/s)^{1/|\alpha|}\right).
	\end{align*}
	Since the integral on the right-hand side is finite by Lemma \ref{Lemma speed A^p}, Lemma \ref{Lemma equiv p(epsilon)} is enough to show that \eqref{Equation to show} holds.
	The convergence in $\mathbb{L}^1$ is just a consequence of Lemma \ref{Lemma asymp expect A(espilon)} and Scheff\'e's Lemma and the proof is complete.
	\end{proof}
	
	\section{Area near the root starting from 0}\label{Section area near the root starting from 0}
	
	We start this section by proving Proposition \ref{Theorem stationary area from 0}.
	
	\begin{proof}[Proof of Proposition \ref{Theorem stationary area from 0}]
	Let $Y^+$ have law $\mathbb{P}_0^+$.
	We identify $\chi_\emptyset$ with $Y^+$ and similarly to Lemma \ref{Lemma Markov-branching property of A}, under $\mathcal{P}_0^+$, we can write
	\begin{align}\label{Equation A as a sum under P_0^+}
		\left\{t^{\frac{\omega_-}{\alpha}}A(t);t\geq 0\right\}
		=\left\{t^{\frac{\omega_-}{\alpha}}\sum_{s\leq t}|\Delta_-Y^+(s)|^{\omega_-}A_s\left((t-s)|\Delta_-Y^+(s)|^\alpha\right);t\geq 0\right\},
	\end{align}
	where $(A_s)_{s\in\mathbb{R}_+}$ is a family of i.i.d. copies of $A$ under $\mathcal{P}_1$, mutually independent from $Y^+$.
	The fact that for all $t\geq 0$, $A(t)<\infty$ $\mathcal{P}_0^+$-a.s. follows from $\sum_{s\leq t}|\Delta_-Y^+(s)|^{\omega_-}<\infty$ $\mathbb{P}_0^+$-a.s. by Lemma 4.3 in \cite{BBCK16} and that each $A_s(t)$ is bounded for all $t\geq 0$ by its total mass $\mathcal{M}_s$ having expectation 1.
	
	Recall the self-similarity property $(Y^+(t))_{t\geq 0}\stackrel{d}{=}(xY^+(tx^\alpha))_{t\geq 0}$ for any $x>0$.
	Combining this fact and the above equation yields, fixing $u\in\mathbb{R}$, that
	\begin{align*}
		&\left\{e^{\frac{\omega_-}{\alpha}(u+t)}A(e^{u+t});t\geq 0\right\}\\
		&\hspace{1.5cm}=\left\{e^{\frac{\omega_-}{\alpha}(u+t)}\sum_{s\leq e^{u+t}}|\Delta_-Y^+(s)|^{\omega_-}A_{s}\left((e^{u+t}-s)|\Delta_-Y^+(se^{\omega_-u})|^\alpha\right);t\geq 0\right\}\\
		&\hspace{1.5cm}=\left\{e^{\frac{\omega_-}{\alpha}(u+t)}\sum_{s\leq e^{t}}|\Delta_-Y^+(se^u)|^{\omega_-}A_{se^u}\left((e^{u+t}-se^u)|\Delta_-Y^+(se^{u})|^\alpha\right);t\geq 0\right\}\\
		&\hspace{1.5cm}\stackrel{d}{=}\left\{e^{\frac{\omega_-}{\alpha}t}\sum_{s\leq e^t}|\Delta_-Y^+(s)|^{\omega_-}A_{se^u}\left((e^t-s)|\Delta_-Y^+(s)|^\alpha\right);t\geq 0\right\}\\
		&\hspace{1.5cm}=\left\{e^{\frac{\omega_-}{\alpha}t}A(e^t);t\geq 0\right\},
	\end{align*}
	which shows the claim.
	\end{proof}
	
	Our goal in the rest of this section is to prove Propositions \ref{Theorem upper envelope A(t)} and \ref{Proposition lower envelope A(t)}.
	
	\subsection{The upper bound}
	
	In this subsection, we do not assume regular variation of $\overline{\Lambda}$ as in \eqref{Equation assumption rho}, but only that $\kappa(\omega_++\omega_-+\alpha)<\infty$.
	We aim at proving Proposition \ref{Theorem upper envelope A(t)}.
	Thanks to Proposition \ref{Theorem stationary area from 0}, we can restrict ourselves without loss of generality to the case where $t\to 0+$.
	We shall need the finiteness of $\mathcal{E}_0^+(t^{\omega_-/\alpha}A(t))$ for all $t>0$, in order to apply Markov's Inequality on some well chosen sequence of events.
	Thanks to Proposition \ref{Theorem stationary area from 0}, it suffices to look at $\mathcal{E}_0^+(A(1))$.
	
	\begin{lem}\label{Lemma E(A(1)) spectrally negative case}
		Suppose that $\kappa(\omega_++\omega_-+\alpha)<\infty$.
		It holds that $\mathcal{E}_0^+(A(1))<\infty$.
	\end{lem}
	
	\begin{proof}
		Thanks to \eqref{Equation A as a sum under P_0^+}, under $\mathcal{P}_0^+$, we have that
		\begin{align*}
		A(1)
		&=\sum_{s\leq 1}|\Delta_-Y^+(s)|^{\omega_-}A_s\left((1-s)|\Delta_-Y^+(s)|^\alpha\right)
		\leq\sum_{s\leq 1}|\Delta_-Y^+(s)|^{\omega_-}\mathcal{M}_s,
		\end{align*}
	where each $\mathcal{M}_s$ is distributed as $\mathcal{M}$ under $\mathcal{P}_1$ and is independent from $Y^+$.
	We can see the above upper bound as the stochastic integral of $s\mapsto \mathcal{M}_s$ with respect to the non-decreasing process $t\mapsto\sum_{s\leq t}|\Delta_-Y^+(s)|^{\omega_-}$ until time 1.
	In particular, we can use the optional projection theorem from \cite{DM2} Theorem 57 (see Theorem 43 of the same book for the definition of the optional projection), which states that the expectation of the stochastic integral is equal to the same expectation where each $\mathcal{M}_s$ has been replaced by its conditional expectation given the natural filtration of $t\mapsto\sum_{s\leq t}|\Delta_-Y^+(s)|^{\omega_-}$.
	In particular, for every negative jump time $s>0$, $\mathcal{E}_0^+(\mathcal{M}_s|(Y^+(u))_{u\leq s})=1$. Therefore,
	\begin{align*}
		\mathcal{E}_0^+\left(A(1)\right)
		&\leq\mathcal{E}_0^+\left(\sum_{s\leq 1}|\Delta_-Y^+(s)|^{\omega_-}\right).
	\end{align*}
	The expected value of the above sum is equal to the expectation of its predictable compensator, given in the proof of Lemma 4.3 in \cite{BBCK16}.
	We thus obtain that the upper bound is equal to
	\begin{align*}
		&\int_{\intervalleoo{-\infty}{0}}\Pi^+(\mathrm{d}y)(1-e^y)^{\omega_-}e^{\omega_+y}\mathcal{E}_0^+\left(\int_0^1Y^+(s)^{\omega_-+\alpha}\mathrm{d}s\right).
	\end{align*}
	The first integral is finite by \eqref{Equation Pi Levy measure of eta^+} and \eqref{Equation Cramer hypothesis}.
	After using Tonelli's Theorem and self-similarity, the other part of the expression becomes
	\begin{align*}
		\int_0^1 s^{\frac{\omega_-}{|\alpha|}-1}\mathrm{d}s\mathcal{E}_0^+\left(Y^+(1)^{\omega_-+\alpha}\right).
	\end{align*}
	Hence, it only remains to check that the expectation is finite.
	If $|\alpha|\leq\omega_-$, then this is the case by Proposition 1(ii) in \cite{BY02}, which ensures that the positive moments of $Y^+(1)$ are finite.
	
	Suppose now that $\alpha<-\omega_-$.
	Theorem 1(iii) in \cite{BY01} shows that
	\begin{align*}
		\mathcal{E}_0^+\left(Y^+(1)^{\omega_-+\alpha}\right)
		&=C\mathbb{E}_1^+\left(I_{\eta^+}^{\omega_-/\alpha}\right),
	\end{align*}
	where $C$ is an explicit constant and
	\begin{align}\label{Equation definition I_eta^+}
		I_{\eta^+}:=\int_0^\infty\exp(\alpha\eta^+(t))\mathrm{d}t
	\end{align}
	(recall that $\eta^+(0)=0$ under $\mathbb{P}_1^+$).
	Proposition 2 in \cite{BY02} gives the finiteness and an expression for the negative moments of $I_{\eta^+}$, under the assumption that $e^{\eta^+(1)}$ admits positive moments of all order.
	However, the proof straightforwardly adapts to $\mathbb{E}_1^+(I_{\eta^+}^{\omega_-/\alpha})$ as soon as $\mathbb{E}_1^+(\exp((\omega_-+\alpha)\eta^+(1)))<\infty$, or equivalently $\kappa(\omega_++\omega_-+\alpha)<\infty$.
	\end{proof}
	
	\begin{proof}[Proof of Proposition \ref{Theorem upper envelope A(t)}]
		Let $\delta>0$. For all $n\geq 1$, we have that
		\begin{align*}
			\mathcal{P}_0^+\left(2^{(n+1)\frac{\omega_-}{|\alpha|}}A(2^{-n})>\log(2^{n+1})^{1+\delta}\right)
			&=\mathcal{P}_0^+\left(A(1)>2^{\frac{\omega_-}{\alpha}}\log(2^{n+1})^{1+\delta}\right)\\
			&\leq\frac{1}{n^{1+\delta}\log(2)^{1+\delta}}\mathcal{E}_0^+(A(1)).
		\end{align*}
		The latter is finite by Lemma \ref{Lemma E(A(1)) spectrally negative case} and therefore summable over $n$.
		Borel-Cantelli's Lemma ensures that $\mathcal{P}_0^+$-a.s. for all $n$ sufficiently large, $A(2^{-n})\leq 2^{(n+1)\frac{\omega_-}{\alpha}}\log(2^{n+1})^{1+\delta}$.
		Since $A$ is non-decreasing, it means that $\mathcal{P}_0^+$-almost surely, for all $n$ large enough and all $t\in\intervallefo{2^{-n-1}}{2^{-n}}$, it holds that
		\begin{align*}
			A(t)\leq t^{\frac{\omega_-}{|\alpha|}}\log(t)^{1+\delta},
		\end{align*}
		which concludes the proof.
	\end{proof}

	\subsection{The lower bound}
	
	Our purpose now is to show Proposition \ref{Proposition lower envelope A(t)}.
	We hence assume that \eqref{Equation assumption rho} holds.
	Similarly as for the upper bound, we can study only the asymptotic behaviour as $t\to\infty$ and easily deduce the one as $t\to 0+$, thanks to Proposition \ref{Theorem stationary area from 0}.
	
	The strategy is to decompose $A(t)$ over the jumps of $Y^+$ as in \eqref{Equation A as a sum under P_0^+}, then to find two functions such that, $\mathcal{P}_0^+$-almost surely, the motion of $Y^+(s)$ is circumscribed in between them for all $s$ large enough.
	
	\subsubsection{Upper and lower envelopes of the Eve cell}
	
	The so-called lower and upper envelopes of positive self-similar Markov processes are described respectively in \cite{CP06} and \cite{P09}.
	We need some preparations to apply the results of these two papers.
	
	Let $U(x)$ be the last passage time of $Y^+$ below $x>0$, that is $U(x):=\sup\left\{t\geq 0:Y^+(t)\leq x\right\}$.
	Define $\nu:=Y^+(U(x)-)/x$ (note that if $\Lambda(\intervalleoo{0}{\infty})=0$ then $\nu=1$ almost surely). 
	When the process admits positive jumps, the law of $\nu$ is given in \cite{CP06} Lemma 1: it has its support included in $\intervalleff{0}{1}$ and satisfies
	\begin{align}\label{Equation law of nu}
		\mathbb{P}(\nu\leq u)=\mathbb{E}(H(1)^{-1})\int_0^1\mathrm{d}v\int_{\intervalleoo{-\frac{\log(u)}{v}}{\infty}}y\Pi_H(\mathrm{d}y/|\alpha|),\qquad u\in\intervalleoo{0}{1},
	\end{align}
	where the subordinator $(H(s))_{s\geq 0}$ is the ascending ladder height process associated with $\eta^+$ and $\Pi_H$ is its L\'evy measure (see Chapter VI of \cite{B96} for background).
	
	The next lemma will be needed to apply the results describing the envelopes of $U$ and $Y^+$.
	Recall the definition of $I_{\eta^+}$ in \eqref{Equation definition I_eta^+}.
	
	\begin{lem}\label{Lemma tails of I}
		\begin{enumerate}[label=(\roman*)]
		\item For all $0<p<(\omega_+-\omega_-)/|\alpha|$, it holds that
		\begin{align*}
			&\mathbb{P}_1(I_{\eta^+}\leq t)=o(t^{p})
			\quad\text{and}\quad \mathbb{P}_1(I_{\eta^+}>1/t)=O(t),
			\qquad\text{as}\quad t\to 0+.
		\end{align*}
		\item Let $q^*:=\min\left\{\omega_+-\omega_-,\sup\left\{p\geq 0:\kappa(\omega_++p)<\infty\right\}\right\}$.
		For all $q<q^*$, it holds that
		\begin{align*}
			\mathbb{P}_1(\nu I_{\eta^+}\leq t)
			=o(t^{q/|\alpha|}),\qquad\text{as}\quad t\to 0+,
		\end{align*}
		where $\nu$ and $I_{\eta^+}$ are independent.
		\end{enumerate}
	\end{lem}
	
	\begin{proof}
	(i) It essentially follows from \cite{R12} Lemma 3 that provides finiteness of some positive and negative moments of $I_{\eta^+}$.
	(Note that Lemma 3 in \cite{R12} gives $\mathbb{E}_1(I_{\eta^+}^{-1})=|\alpha\kappa'(\omega_+)|$, which is indeed finite under our assumption \eqref{Equation Cramer hypothesis}.)
	
	(ii)
	For all $u\in\intervalleoo{0}{1}$, equation \eqref{Equation law of nu} entails that
	\begin{align}\label{Equation left tail nu}
		\mathbb{P}(\nu\leq u)
		&\leq\mathbb{E}(H(1)^{-1})|\alpha|\int_{\intervalleoo{-\log(u)/|\alpha|}{\infty}}y\Pi_H(\mathrm{d}y)\nonumber\\
		&=\mathbb{E}(H(1)^{-1})|\alpha|\left(-\frac{\log(u)}{|\alpha|}\overline{\Pi}_H(-\log(u)/|\alpha|)+\int_{-\log(u)/|\alpha|}^{\infty} \overline{\Pi}_H(z)\mathrm{d}z\right),
	\end{align}
	by Tonelli's Theorem.
	Theorem 7.8 in \cite{K14} shows that the right tail of $\Pi_H$ is described for all $y>0$ by
	\begin{align*}
		\overline{\Pi}_H(y)
		&=\int_{\intervallefo{0}{\infty}}\widehat{U}(\mathrm{d}z)\overline{\Pi}^+(z+y),
	\end{align*}
	where $\widehat{U}$ is the renewal measure of the descending ladder height process of $\eta^+$.
	Note that the integral is finite, see e.g. \cite{K14} Corollary 5.3.
	By assumption \eqref{Equation Cramer hypothesis}, there exists $q>0$ such that $q<\sup\left\{p\geq 0:\kappa(\omega_++p)<\infty\right\}$. In particular, $\kappa(\omega_++q)<\infty$, which is equivalent to 
	\begin{align*}
		\int_{\intervalleoo{1}{\infty}} e^{qy}\Pi^+(\mathrm{d}y)
		=\overline{\Pi}^+(1)+q\int_0^\infty e^{qx}\overline{\Pi}^+(x)\mathrm{d}x
		<\infty,
	\end{align*}
	so in particular $\overline{\Pi}^+(x)=o(e^{-qx})$ as $x\to\infty$.
	This leads to
	\begin{align*}
		\overline{\Pi}_H(y)
		&\leq e^{-qy}\int_{\intervallefo{0}{\infty}}\widehat{U}(\mathrm{d}z)e^{-qz}=O(e^{-qy})
	\end{align*}
	Coming back to \eqref{Equation left tail nu}, we see that
	\begin{align*}
		\mathbb{P}(\nu\leq u)
		&=O\left(|\log(u)|u^{q/|\alpha|}+\int_{0}^{u} \overline{\Pi}_H(-\log(y)/|\alpha|)\frac{\mathrm{d}y}{y}\right)
		=O\left(|\log(u)|u^{q/|\alpha|}+\int_{0}^{u}y^{q/|\alpha|-1}\mathrm{d}y\right)\\
		&=O\left(|\log(u)|u^{q/|\alpha|}\right).
	\end{align*}
	For $\delta>0$ small enough, the same reasonning holds with $q$ replaced by $q+\delta$, so that $\mathbb{P}(\nu\leq u)=o(u^{q/|\alpha|})$.
	
	To conclude, suppose that $X_1$ and $X_2$ are independent positive random variables such that for $i=1,2$, $\mathbb{P}(X_i\leq t)=o(t^{q_i})$ as $t\to0+$, for some $q_1>q_2>0$.
	Note that there exists a constant $C$ such that $\mathbb{P}(X_2\leq t)\leq Ct^{q_2}$ for all $t\geq 0$.
	Then, we write
	\begin{align*}
		\mathbb{P}_1(X_1X_2\leq t)
		&\leq\int_{\mathbb{R}_+}\mathbb{P}(X_1\in\mathrm{d}x)\mathbb{P}(X_2\leq t/x)
		\leq\int_{\mathbb{R}_+}\mathbb{P}(X_1\in\mathrm{d}x)C(t/x)^{q_2}
		=O(t^{q_2}),
	\end{align*}
	since $q_1>q_2$.
	The bound for $\mathbb{P}_1(\nu I_{\eta^+}\leq t)$ as $t\to 0+$ thus follows from this fact and part (i).
	\end{proof}

	Let $\underline{p}>1/|\alpha|$ and $\overline{p}>1/q^*$ where $q^*$ is defined in Lemma \ref{Lemma tails of I}.
	Define on $\intervalleoo{1}{\infty}$ the functions
	\begin{align}\label{Equation definition of g^arrow general}
		&g^{\uparrow}:t\mapsto
		t^{1/|\alpha|}\log(t)^{\overline{p}},
		\hspace{1cm} g^{\downarrow}:t\mapsto
		t^{1/|\alpha|}\log(t)^{-\underline{p}}.
	\end{align}
	Note that $g^{\uparrow}$ and $g^{\downarrow}$ are regularly varying at $\infty$ with index $1/|\alpha|$.
	Let $\overline{g}(t):=\sup_{s\leq t}g^{\uparrow}(s)$ and similarly, let $\underline{g}(t):=\inf_{s\geq t}g^{\downarrow}(s)$.
	These functions have the nice property to be strictly increasing and by Theorem 1.5.3. in \cite{BGT87} we have that
	\begin{align}\label{Equation monotone equivalents}
		\overline{g}(t)\underset{t\to \infty}{\sim} g^{\uparrow}(t)\quad\text{and}\quad\underline{g}(t)\underset{t\to \infty}{\sim} g^{\downarrow}(t).
	\end{align}
	
	\begin{lem}\label{Lemma envelopes Y and U}
		$\mathbb{P}_0^+$-almost surely, there exists $s>0$ such that
		\begin{align*}
			\underline{g}(t)
			&<Y^+(t)
			<\overline{g}(t),\hspace{1.05cm}\forall t\in\intervalleoo{s}{\infty}.
		\end{align*}
	\end{lem}

	\begin{proof}
		First, note that $Y^+(t)<x$ implies $U(x)>t$.
		Let $T(1):=\inf\{t\geq 0:Y^+(t)\geq 1\}$ and note that $T(1)\leq U(1)$. Proceeding as in the proof of Proposition 6 in \cite{P09}, one can show that
		\begin{align*}
			\mathbb{P}(T(1)<t)\leq C\mathbb{P}(\nu I_{\eta^+}\leq t).
		\end{align*}
		Lemma \ref{Lemma tails of I}(ii) in this paper and Proposition 4 in \cite{P09} then show the claim for $\overline{g}$.
		The statement involving $\underline{g}$ is shown using Theorem 1(i) in \cite{CP06} and our Lemma \ref{Lemma tails of I}(i).
	\end{proof}
	
	\subsubsection{Proof of Proposition \ref{Proposition lower envelope A(t)}}
	
	For clarity purpose, we prove two Lemmas that will make the proof of Proposition \ref{Proposition lower envelope A(t)} straightforward.
	We need some notation.
	Define
	\begin{align}\label{Equation definition f}
		f(t):=t\log(t)^{\alpha\overline{p}}.
	\end{align}
	We note for later use that \eqref{Equation definition of g^arrow general} and \eqref{Equation monotone equivalents} implies that
	\begin{align}\label{Equation asymptotic g(f)}
		\underline{g}(f(t))&\underset{t\to\infty}{\sim}t^{\frac{1}{|\alpha|}}\log(t)^{-\underline{p}-\overline{p}},
		\hspace{1cm}\overline{g}(f(t))\underset{t\to\infty}{\sim}t^{\frac{1}{|\alpha|}}.
	\end{align}
	To control the fluctuations of $Y^+$, we work on the event
	\begin{align*}
		E_t:=&\left\{\forall s\geq f(t):\underline{g}(s)<Y^+(s)<\overline{g}(s)\right\}.
	\end{align*}
	Note that $f$ being eventually increasing, Lemma \ref{Lemma envelopes Y and U} ensures that $\mathbb{P}_0^+(\liminf_{t\to\infty}E_t)=1$.
	Moreover, noting that $f(t)/t\to 0$ as $t\to\infty$, we see that $J_t:=\intervalleff{f(t)}{f(2t)}\subset\intervalleff{0}{t}$.
		
	\begin{lem}\label{Lemma lower bound A}
		For all $a>0$, $\mathcal{P}_0^+$-almost surely, it holds on the event $E_t$ that
		\begin{align*}
			A(t)\geq t^{\frac{\omega_-}{|\alpha|}}\log(t)^{-\omega_-(a+\underline{p}+\overline{p})}\sum_{s\in J_t}A_s(1/3)\mathds{1}_{\left\{\Delta_-\eta^+(\tau_s)\leq \log(1-\log(t)^{^{-a}})\right\}},
		\end{align*}
		for all $t$ large enough.
	\end{lem}
	
	\begin{proof}
	We need the lower envelope of $Y^+$ to ensure that the jumps are large enough, and the upper envelope of $Y^+$ for the area generated from it to be not too small.
	On the event $E_t$, we have
	\begin{align*}
		A(t)
		&=\sum_{s\leq t}|\Delta_-Y^+(s)|^{\omega_-}A_s\left((t-s)|\Delta_-Y^+(s)|^{\alpha}\right)\\
		&\geq\sum_{s\in J_t}|\Delta_-Y^+(s)|^{\omega_-}A_s\left((t-s)|\Delta_-Y^+(s)|^{\alpha}\right)\\
		&\geq\sum_{s\in J_t}\underline{g}(s)^{\omega_-}(1-e^{\Delta_-\eta^+(\tau_s)})^{\omega_-}A_s\left((t-s)\overline{g}(s)^{\alpha}(1-e^{\Delta_-\eta^+(\tau_s)})^{\alpha}\right)\\
		&\geq\sum_{s\in J_t}\underline{g}(f(t))^{\omega_-}(1-e^{\Delta_-\eta^+(\tau_s)})^{\omega_-} A_s\left(\left(t-f(2t)\right)\overline{g}(f(2t))^{\alpha}\left(1-e^{\Delta_-\eta^+(\tau_s)}\right)^{\alpha}\right).
	\end{align*}
	Note that $f(2t)/t\to 0$ and \eqref{Equation asymptotic g(f)} entails that $t\overline{g}(f(2t))^{\alpha}\to 1/2$ as $t\to\infty$.
	Hence, taking $t$ large enough, we see for each $s\in J_t$,
	\begin{align*}
		A_s\left(\left(t-f(2t)\right)\overline{g}(f(2t))^{\alpha}\left(1-e^{\Delta_-\eta^+(\tau_s)}\right)^{\alpha}\right)\geq A_s(1/3).
	\end{align*}
	We write
	\begin{align*}
		A(t)
		&\geq\underline{g}(f(t))^{\omega_-}\sum_{s\in J_t}(1-e^{\Delta_-\eta^+(\tau_s)})^{\omega_-}A_s\left(1/3\right)\mathds{1}_{\left\{\Delta_-\eta^+(\tau_s)<\log(1-\log(t)^{-a})\right\}}\\
		&\geq\underline{g}(f(t))^{\omega_-}\log(t)^{-a\omega_-}\sum_{s\in J_t}A_s\left(1/3\right)\mathds{1}_{\left\{\Delta_-\eta^+(\tau_s)<\log(1-\log(t)^{-a})\right\}}\\
		&\!\!\!\underset{t\to\infty}{\sim}t^{\frac{\omega_-}{|\alpha|}}\log(t)^{-\omega_-(a+\underline{p}+\overline{p})}\sum_{s\in J_t}A_s\left(1/3\right)\mathds{1}_{\left\{\Delta_-\eta^+(\tau_s)<\log(1-\log(t)^{-a})\right\}},
	\end{align*}
	where we used \eqref{Equation asymptotic g(f)}.
	\end{proof}
	
	For the lower bound of Lemma \ref{Lemma lower bound A} to be useful, we need to make sure that the sum contains at least one non negligeable term.
	We shall use the following:
	
	\begin{lem}\label{Lemma bound number of jumps}
		For all $a>\max\{0,\frac{|\alpha|}{\rho}(\overline{p}-\underline{p})\}$, $\mathcal{P}_0^+$-almost surely, for all $t$ large enough, the number of elements in $\left\{s\in J_t:\Delta_-\eta^+(\tau_s)\leq\log(1-\log(t)^{-a})\right\}$ is stochastically bounded by a Poisson random variable with parameter
		\begin{align*}
			\frac{1}{2}\log(t)^{|\alpha|(\underline{p}-\overline{p})+a\rho}.
		\end{align*}
	\end{lem}
	
	\begin{proof}
	By definition of the Lamperti time-change $s\mapsto\tau_s$, we have that
	\begin{align*}
		f(2t)-f(t)
		=\int_{\tau_{f(t)}}^{\tau_{f(2t)}}e^{|\alpha|\eta^+(s)}\mathrm{d}s,
	\end{align*}
	so in particular, on the event $E_t$ such that $\eta^+(\tau_s)<\log(\underline{g}(2t))$ for all $s\in J_t$, we get that
	\begin{align*}
		\left(f(2t)-f(t)\right)\underline{g}(2t)^{\alpha}\leq\tau_{f(2t)}-\tau_{f(t)}
	\end{align*}
	Therefore, changing the variables, the domain of the sum with respect to $u=\tau_s$ becomes at least of length
	\begin{align*}
		\frac{f(2t)-f(t)}{\underline{g}(2t)^{|\alpha|}}
		&\underset{t\to\infty}{\sim}\frac{2t|\log(2t)|^{\alpha\overline{p}}-t|\log(t)|^{\alpha\overline{p}}}{2t|\log(2t)|^{\alpha\underline{p}}}
		\underset{t\to\infty}{\sim}\frac{1}{2}|\log(t)|^{\alpha(\overline{p}-\underline{p})}
	\end{align*}
	Hence, $\eta^+$ being a L\'evy process, the number of $s\in J_t$ such that $\mathds{1}_{\left\{\Delta_-\eta^+(\tau_s)<\log(1-\log(t)^{-a})\right\}}=1$ is greater than a Poisson random variable with parameter
	\begin{align*}
		&\frac{1}{2}|\log(t)|^{\alpha(\overline{p}-\underline{p})}\Pi^+\left(\intervalleoo{-\infty}{\log(1-\log(t)^{-a})}\right)\\
		&\hspace{2cm}\underset{t\to\infty}{\sim}\frac{1}{2}|\log(t)|^{\alpha(\overline{p}-\underline{p})}\Pi^+\left(\intervalleoo{-\infty}{-\log(t)^{-a}}\right)
		\geq\frac{1}{2}|\log(t)|^{\alpha(\overline{p}-\underline{p})+a\rho},
	\end{align*}
	where we used \eqref{Equation Pi Levy measure of eta^+} to write
	\begin{align*}
		\Pi^+\left(\intervalleoo{-\infty}{-\log(t)^{-a}}\right)
		&\geq\int_{\intervalleoo{-\infty}{-\log(t)^{-a}}}e^{\omega_+y}\Lambda(\mathrm{d}y)
		=O\left(\log(t)^{a\rho}\right),
	\end{align*}
	as $t\to\infty$, by \eqref{Equation assumption rho} and Theorem 1.6.4 of \cite{BGT87}.
	\end{proof}
	
	\begin{proof}[Proof of Proposition \ref{Proposition lower envelope A(t)}]
	
	Recall that Lemma \ref{Lemma envelopes Y and U}, ensures that $\mathbb{P}_0^+(\liminf_{t\to\infty}E_t)=1$.
	In view of Lemma \ref{Lemma lower bound A}, it suffices to establish that the sum in its lower bound is almost surely bounded away from 0 for all $t$ large enough.
	
	Choose $a>\max\{0,\frac{|\alpha|}{\rho}(\overline{p}-\underline{p})\}$ and, to ease the notation, let $\delta:=\alpha(\overline{p}-\underline{p})+a\rho$ and note that $\delta>0$.
	Let $n\geq 1$.
	Thanks to Lemma \ref{Lemma bound number of jumps}, we can write
	\begin{align*}
		&\mathcal{P}_0^+\left(\sum_{s\in J_{2^n}}A_{s}\left(1/3\right)\mathds{1}_{\left\{\Delta_-\eta^+(\tau_s)<\log\left(1-\log(2^n)^{-a}\right)\right\}}\leq 1\right)\\
		&\hspace{5cm}= O\left(e^{-\log(2^n)^{\delta}}\sum_{m\geq 0}\frac{\log(2^n)^{\delta m}}{m!}\mathcal{P}_1(A(1/3)\leq 1)^m\right)\\
		&\hspace{5cm}= \exp\left(-n^{\delta}\log(2)^{\delta}\mathcal{P}_1(A(1)>1/3)\right).
	\end{align*}
	The latter being summable, Borel-Cantelli's Lemma shows that $\mathcal{P}_0^+$- almost surely, there exists $n_0\geq 1$ such that for all $n\geq n_0$, it holds that
	\begin{align*}
		\sum_{s\in J_{2^n}}A_{s}\left(1/3\right)\mathds{1}_{\left\{\Delta_-\eta^+(\tau_s)<\log\left(1-\log(2^n)^{-a}\right)\right\}}>1.
	\end{align*}
	This, the facts that $\bigcup_{n\geq n_0}J_{2^{n}}=\intervallefo{f(2^{n_0})}{\infty}$, that $A$ is non-decreasing and Lemma \ref{Lemma lower bound A} together imply that on the event $E_t$, $\mathcal{P}_0^+$-almost surely it holds that
	\begin{align}\label{Equation exponent lower envelope}
		A(t)
		&\geq t^{\frac{\omega_-}{|\alpha|}}\log(t)^{-\omega_-(a+\underline{p}+\overline{p})},
	\end{align}
	which ensures the claim.
	\end{proof}
	
	\begin{rem}\label{Remark exponent lower envelope}
	Although we do not claim that it is optimal, Equation \eqref{Equation exponent lower envelope} shows that Proposition \ref{Proposition lower envelope A(t)} applies for all $q>\omega_-(1/|\alpha|+1/q^*+\max\{0,\frac{|\alpha|}{\rho}(1/q^*-1/|\alpha|)\})$, where $q^*$ is defined in Lemma \ref{Lemma tails of I}.
	Indeed, we have respectively chosen $\underline{p}$ arbitrary close from above to $1/\alpha$ and $\overline{p}$ arbitrary close from above to $1/q^*$.
	\end{rem}

	\section{Application to random maps}\label{Section Application to random maps}
	
	\paragraph{A specific family of growth-fragmentations.}
	We start this section by recalling the connection between growth-fragmentations and random surfaces that has been observed in \cite{BCK15} for Boltzmann triangulations approximating Brownian disks, and that has been generalised in \cite{BBCK16} to a broader family of Boltzmann maps approximating stable disks and plane.
	
	More details on what follows can be found in \cite{BBCK16}.
	Let $\theta\in\intervalleof{1}{3/2}$ and for all $q\in\intervalleoo{\theta}{2\theta+1}$, let 
	\begin{align}\label{Equation kappa_theta}
		\kappa_\theta(q):=\frac{\cos(\pi(q-\theta))}{\sin(\pi(q-2\theta))}\cdot\frac{\Gamma(q-\theta)}{\Gamma(q-2\theta)}.
	\end{align}
	Thus defined, $\kappa_\theta$ is the cumulant of a specific self-similar growth-fragmentation $\mathbf{X}_\theta$. Let its index of self-similarity be $\alpha=1-\theta$. 
	
	Informally, the collection of cycles' lengths observed at heights in some discrete random maps with large boundary converges, when properly rescaled, towards $\mathbf{X}_\theta$, where $\theta$ depends on the tail of the distribution of the degree of a typical face (see Theorem 6.8 in \cite{BBCK16}).
	The Brownian case corresponds to $\theta=3/2$.
	This means that we obtain $\mathbf{X}_{3/2}$ under $\mathcal{P}_1$ (respectively under $\mathcal{P}_0^+$) in the scaling limit of the sliced approximation of the free\footnote{\textit{free} refers to the fact that the total intrinsic area of the surface is not fixed but random, with law described in \cite{BC17} Proposition 4.} Brownian disk (respectively plane); As noted in introduction, $\mathbf{X}_{3/2}$ also appears when directly slicing the Brownian disk or the Brownian plane (see \cite{L17} Theorem 3 and 23).
	
	In the case of $\mathbf{X}_\theta$, $\eta^+$ and $\eta^-$ belong to the class of hypergeometric L\'evy processes, see Proposition 5.2 in \cite{BBCK16} for this fact and \cite{KP13} for a definition and references on hypergeometric L\'evy processes.
	This yields an explicit expression of the densities of $\Pi^+$ and $\Pi^-$.
	
	\begin{lem}\label{Lemma integrability Pi^+}
		Let $c_-:=\frac{\Gamma(\theta+1)}{\pi}$ and $c_+:=\frac{\Gamma(\theta+1)}{\pi}\sin(\pi(\theta-1/2))$.
		The densities of $\Pi^+$ and $\Pi^-$ are given by
		\begin{align*}
			&\Pi^+(\mathrm{d}y)/\mathrm{d}y
			=c_-\frac{e^{3y/2}}{(1-e^y)^{\theta+1}}\mathds{1}_{\left\{y<0\right\}}+c_+\frac{e^{3y/2}}{(e^y-1)^{\theta+1}}\mathds{1}_{\left\{y>0\right\}},\\
			&\Pi^-(\mathrm{d}y)/\mathrm{d}y
			=c_-\frac{e^{y/2}}{(1-e^y)^{\theta+1}}\mathds{1}_{\left\{y<0\right\}}+c_+\frac{e^{y/2}}{(e^y-1)^{\theta+1}}\mathds{1}_{\left\{y>0\right\}}.
		\end{align*}
	\end{lem}
	\begin{proof}
		In the proof of Proposition 5.2 in \cite{BBCK16}, it is shown that the density $h$ of the image of $\Pi^+$ by $x\mapsto e^x$ is given by
		\begin{align*}
			h(z)
			&=\frac{\Gamma(\theta+1)}{\pi}\frac{z^{1/2}}{(1-z)^{\theta+1}}\mathds{1}_{\left\{0<z<1\right\}}+\frac{\Gamma(\theta+1)\sin(\pi(\theta-1/2))}{\pi}\frac{z^{1/2}}{(z-1)^{\theta+1}}\mathds{1}_{\left\{z>1\right\}}.
		\end{align*}
		The expression of $\Pi^+(\mathrm{d}y)/\mathrm{d}y$ follows from a straightforward change of variables.
		One then gets the expression of $\Pi^-(\mathrm{d}y)/\mathrm{d}y$ from \cite{CKP09} Section 2 by identifying $c_-$ and $c_+$.

	\end{proof}
	
	From \eqref{Equation kappa_theta}, one sees that \eqref{Equation Cramer hypothesis} is satisfied with $\omega_-=\theta+1/2$ and $\omega_+=\theta+3/2$.
	Moreover, one can check that the assumption \eqref{Equation assumption rho} is also satisfied with $\rho=\theta$, by looking at the behaviour of $\Pi^-(\intervalleoo{-\infty}{\log(x)})$, as explained after \eqref{Equation tail Pi^-}.
	This means that our results apply in particular to $\mathbf{X}_\theta$.
	To sum up, the parameters in terms of $\theta$ are
	\begin{align}\label{Equation parameters in terms of theta}
	\begin{cases}
		\alpha&=1-\theta\\
		\omega_-&=\theta+1/2\\
		\omega_+&=\theta+3/2\\
		\rho&=\theta
	\end{cases}
	\end{align}

	\paragraph{Area of a a small annular in Brownian and stable disks.}
	Le Gall showed the following in \cite{L17} Theorem 3: denote $\mathbf{v}$ the intrinsic area measure of the free Brownian disk with boundary size $x>0$, and let $B_\epsilon$ be the annular of width $\epsilon$ from the boundary (the set of points at distance smaller than $\epsilon$ from the boundary). Then it holds that
	\begin{align}\label{Equation Le Gall}
		\lim_{\epsilon\to 0+}\epsilon^{-2}\mathbf{v}(B_\epsilon)=x,\quad\text{almost surely.}
	\end{align}
	As explained in the beginning of the section, the above convergence can be translated in terms of $\mathbf{X}_{3/2}$ under $\mathcal{P}_x$, where the area of the annular corresponds to $A(\epsilon)$.
	We check whether we retrieve the same result using our theorem \ref{Theorem area GF}.
	
	Consider $\mathbf{X}_{\theta}$ for $\theta\in\intervalleof{1}{3/2}$.
	The computations before Theorem \ref{Theorem area GF} show that
	\begin{align*}
		&\overline{\Lambda}(\epsilon)
		=\Lambda(\intervalleoo{-\infty}{-\epsilon})
		\underset{\epsilon\to 0+}{\sim}\frac{1/2}{\theta}\epsilon^{-\theta-1/2}\Pi^-\left(\intervalleoo{-\infty}{\log(\epsilon)}\right),\\
	\end{align*}
	which thanks to Lemma \ref{Lemma integrability Pi^+} can be written as $\overline{\Lambda}(\epsilon)\sim\frac{c_-}{\theta}\epsilon^{\theta}$.
	Recalling \eqref{Equation parameters in terms of theta}, Theorem \ref{Theorem area GF} thus reads as follows:	
	For all $x>0$, it holds that
	\begin{align*}
			\lim_{\epsilon\to 0+}\epsilon^{-\frac{\theta-1/2}{\theta-1}}A(\epsilon)=\frac{2(\theta-1)}{(\theta-1/2)}\cdot c_-\mathbb{E}_1^-\left(I^{-\frac{1}{2(\theta-1)}}\right)x,\quad\mathcal{P}_x\text{-a.s. and in }\mathbb{L}^1.
	\end{align*}
	Note that for $\theta=3/2$, we get
\begin{align*}
	\lim_{\epsilon\to 0+}\epsilon^{-2}A(\epsilon)
	&=c_-\mathbb{E}_1^-\left(I^{-1}\right)x
	=\frac{\Gamma(5/2)}{2\pi}|\kappa_{3/2}'(2)|x,
\end{align*}
where the value of the expectation is given in \cite{CPY97} Proposition 3.1(iv). (Note that since $I=\int_0^\infty\exp(\eta^-(t)/2)\mathrm{d}t$, we get $\mathbb{E}_1^-(I^{-1})=|\mathbb{E}_1^-(\eta^-(1)/2)|$.)
Because $\kappa_{3/2}(2)=0$ and using the explicit expression of $\kappa_{3/2}$ given in \eqref{Equation kappa_theta}, we have that
\begin{align*}
	\kappa_{3/2}'(2)
	&=\lim_{q\to 2}\frac{\Gamma(q-3/2)}{(q-2)\Gamma(q-3)}
	=\frac{\Gamma(1/2)}{\text{Res}_{\Gamma}(-1)}
	=-\sqrt{\pi},
\end{align*}
where we used that the residue of $\Gamma$ at $-1$ is $-1$ and $\Gamma(1/2)=\sqrt{\pi}$.
Well known properties of the gamma function entails that $\Gamma(5/2)=3\Gamma(1/2)/4=3\sqrt{\pi}/4$.
We then obtain
\begin{align*}
	\epsilon^{-2}A(\epsilon)\underset{\epsilon\to 0+}{\sim}\epsilon^2\frac{3}{8}x.
\end{align*}
	Note that we get an additional factor $3/8$ compare to \eqref{Equation Le Gall}.
	The cumulant function $\kappa$ in \cite{LR18} Section 11.1 is equal to $\sqrt{8/3}\cdot\kappa_{3/2}$, which simply corresponds to multiplying the distance in the Brownian map by a constant.
	This means in particular, considering this $\kappa$ instead of ours, that $\eta^-$ becomes $\sqrt{8/3}\cdot \eta^-$, and the constant $c_-$ of Lemma \ref{Lemma integrability Pi^+} becomes $\sqrt{8/3}\cdot c_-$.
	The above is therefore consistent with \eqref{Equation Le Gall}.
	
	\paragraph{Area of a small ball in Brownian and stable maps.}
	If we read Proposition \ref{Theorem stationary area from 0} in terms of $\mathbf{X}_\theta$, we get that the law of $t^{(\theta+1/2)/(1-\theta)}A(t)$ is the same for all $t>0$. This coincides with the volume growth exponent of infinite Boltzmann planar maps in the so-called \textit{dilute phase}, meaning that the number of vertices at height $n$ as $n\to\infty$ in such maps grows at a speed $n^{(\theta+1/2)/(1-\theta)}$, see equation (4.3) in Theorem 4.2 of \cite{BC17}.
	
	When $\theta=3/2$, this exponent is equal to $-4$.
	In \cite{L07} Lemma 6.2 (see also \cite{M14} Lemma 4.4.4), it is shown that in this case, for any $\delta>0$, it holds that
	\begin{align*}
			\limsup_{t\to 0+}t^{-4+\delta}A(t)=0,\quad\mathbb{P}_0^+\text{-a.s.}
	\end{align*}
	The assumption $\kappa(\omega_++\omega_-+\alpha)<\infty$ in Proposition \ref{Theorem upper envelope A(t)} reads as $\kappa_{3/2}(5/2)<\infty$, which is indeed the case by \eqref{Equation kappa_theta}.
	In particular, our propositions \ref{Theorem upper envelope A(t)} and \ref{Proposition lower envelope A(t)} thus improve the above, replacing $t^{\delta}$ by a power of $|\log(t)|$ and considering the $\liminf$ as well.
	The exponent $q_0$ on the $\log$ of the lower bound is given in Remark \ref{Remark exponent lower envelope}, in terms of some constant $q^*$ defined in Lemma \ref{Lemma tails of I}.
	One can check that for any $\theta\in\intervalleof{1}{3/2}$, one has that $q^*=\min\{1,\theta-1/2\}=\theta-1/2$, and therefore $q_0=(\theta+1/2)\left((\theta-1)^{-1}+(\theta-1/2)^{-1}\right)$.
	In the Brownian case $\theta=3/2$, we obtain $q_0=6$.
	
	Thus, for $\mathbf{X}_\theta$ with any $\theta\in\intervalleof{1}{3/2}$, the following holds:
	for all $\delta>0$ and $q>q_0$ as above, we have that
		\begin{align*}
			&\limsup_{t\to 0\text{ or }\infty}|\log(t)|^{-1-\delta}t^{\frac{\theta+1/2}{1-\theta}}A(t)=0,\quad\mathbb{P}_0^+\text{-a.s.}\\
			&\liminf_{t\to 0\text{ or }\infty}|\log(t)|^{q}t^{\frac{\theta+1/2}{1-\theta}}A(t)=\infty,\quad\mathbb{P}_0^+\text{-a.s.}
		\end{align*}

	\bibliographystyle{plain}
	\bibliography{biblio}	

\begin{thebibliography}{10}

\bibitem{AR15}
Jonas Arista and V\'ictor~M. Rivero.
\newblock Implicit renewal theory for exponential functionals of lévy
  processes.
\newblock {\em arXiv:1510.01809}, 2015.

\bibitem{B96}
Jean Bertoin.
\newblock {\em L\'evy processes}, volume 121 of {\em Cambridge Tracts in
  Mathematics}.
\newblock Cambridge University Press, Cambridge, 1996.

\bibitem{B01}
Jean Bertoin.
\newblock Homogeneous fragmentation processes.
\newblock {\em Probab. Theory Related Fields}, 121(3):301--318, 2001.

\bibitem{B04}
Jean Bertoin.
\newblock On small masses in self-similar fragmentations.
\newblock {\em Stochastic Process. Appl.}, 109(1):13--22, 2004.

\bibitem{B17}
Jean Bertoin.
\newblock Markovian growth-fragmentation processes.
\newblock {\em Bernoulli}, 23(2):1082--1101, 2017.

\bibitem{BBCK16}
Jean Bertoin, Timothy Budd, Nicolas Curien, and Igor Kortchemski.
\newblock Martingales in self-similar growth-fragmentations and their
  connections with random planar maps.
\newblock {\em Probab. Theory Related Fields}, 172(3-4):663--724, 2018.

\bibitem{BCK15}
Jean Bertoin, Nicolas Curien, and Igor Kortchemski.
\newblock Random planar maps and growth-fragmentations.
\newblock {\em Ann. Probab.}, 46(1):207--260, 2018.

\bibitem{BLM08}
Jean Bertoin, Alexander Lindner, and Ross Maller.
\newblock On continuity properties of the law of integrals of {L}\'evy
  processes.
\newblock In {\em S\'eminaire de probabilit\'es {XLI}}, volume 1934 of {\em
  Lecture Notes in Math.}, pages 137--159. Springer, Berlin, 2008.

\bibitem{BY01}
Jean Bertoin and Marc Yor.
\newblock On subordinators, self-similar {M}arkov processes and some
  factorizations of the exponential variable.
\newblock {\em Electron. Comm. Probab.}, 6:95--106, 2001.

\bibitem{BY02}
Jean Bertoin and Marc Yor.
\newblock On the entire moments of self-similar {M}arkov processes and
  exponential functionals of {L}\'{e}vy processes.
\newblock {\em Ann. Fac. Sci. Toulouse Math. (6)}, 11(1):33--45, 2002.

\bibitem{BM17}
J\'{e}r\'{e}mie Bettinelli and Gr\'{e}gory Miermont.
\newblock Compact {B}rownian surfaces {I}: {B}rownian disks.
\newblock {\em Probab. Theory Related Fields}, 167(3-4):555--614, 2017.

\bibitem{BGT87}
N.~H. Bingham, C.~M. Goldie, and J.~L. Teugels.
\newblock {\em Regular variation}, volume~27 of {\em Encyclopedia of
  Mathematics and its Applications}.
\newblock Cambridge University Press, Cambridge, 1987.

\bibitem{BC17}
Timothy Budd and Nicolas Curien.
\newblock Geometry of infinite planar maps with high degrees.
\newblock {\em Electron. J. Probab.}, 22:Paper No. 35, 37, 2017.

\bibitem{CPY97}
Philippe Carmona, Fr\'ed\'erique Petit, and Marc Yor.
\newblock On the distribution and asymptotic results for exponential
  functionals of {L}\'evy processes.
\newblock In {\em Exponential functionals and principal values related to
  {B}rownian motion}, Bibl. Rev. Mat. Iberoamericana, pages 73--130. Rev. Mat.
  Iberoamericana, Madrid, 1997.

\bibitem{CKP09}
L.~Chaumont, A.~E. Kyprianou, and J.~C. Pardo.
\newblock Some explicit identities associated with positive self-similar
  {M}arkov processes.
\newblock {\em Stochastic Process. Appl.}, 119(3):980--1000, 2009.

\bibitem{CP06}
Loic Chaumont and J.~C. Pardo.
\newblock The lower envelope of positive self-similar {M}arkov processes.
\newblock {\em Electron. J. Probab.}, 11:no. 49, 1321--1341, 2006.

\bibitem{CL14}
Nicolas Curien and Jean-Fran\c{c}ois Le~Gall.
\newblock The {B}rownian plane.
\newblock {\em J. Theoret. Probab.}, 27(4):1249--1291, 2014.

\bibitem{DM2}
Claude Dellacherie and Paul-Andr\'e Meyer.
\newblock {\em Probabilit\'es et potentiel. {C}hapitres {V} \`a {VIII}}, volume
  1385 of {\em Actualit\'es Scientifiques et Industrielles [Current Scientific
  and Industrial Topics]}.
\newblock Hermann, Paris, revised edition, 1980.
\newblock Th\'eorie des martingales. [Martingale theory].

\bibitem{G18}
Fran\c{c}ois~G. Ged.
\newblock Profile of a self-similar growth-fragmentation.
\newblock {\em Electron. J. Probab.}, 24:Paper No. 7, 21, 2019.

\bibitem{KP13}
A.~Kuznetsov and J.~C. Pardo.
\newblock Fluctuations of stable processes and exponential functionals of
  hypergeometric {L}\'{e}vy processes.
\newblock {\em Acta Appl. Math.}, 123:113--139, 2013.

\bibitem{K14}
Andreas~E. Kyprianou.
\newblock {\em Fluctuations of {L}\'evy processes with applications}.
\newblock Universitext. Springer, Heidelberg, second edition, 2014.
\newblock Introductory lectures.

\bibitem{L07}
Jean-Fran\c{c}ois Le~Gall.
\newblock The topological structure of scaling limits of large planar maps.
\newblock {\em Invent. Math.}, 169(3):621--670, 2007.

\bibitem{L13}
Jean-Fran\c{c}ois Le~Gall.
\newblock Uniqueness and universality of the {B}rownian map.
\newblock {\em Ann. Probab.}, 41(4):2880--2960, 2013.

\bibitem{L17}
Jean-Fran\c{c}ois Le~Gall.
\newblock Brownian disks and the {B}rownian snake.
\newblock {\em Ann. Inst. Henri Poincar\'{e} Probab. Stat.}, 55(1):237--313,
  2019.

\bibitem{LM10}
Jean-Fran\c{c}ois Le~Gall and Gr\'{e}gory Miermont.
\newblock On the scaling limit of random planar maps with large faces.
\newblock In {\em X{VI}th {I}nternational {C}ongress on {M}athematical
  {P}hysics}, pages 470--474. World Sci. Publ., Hackensack, NJ, 2010.

\bibitem{LM11}
Jean-Fran\c{c}ois Le~Gall and Gr\'egory Miermont.
\newblock Scaling limits of random planar maps with large faces.
\newblock {\em Ann. Probab.}, 39(1):1--69, 2011.

\bibitem{LR18}
Jean-Fran\c{c}ois Le~Gall and Armand Riera.
\newblock Growth-fragmentation processes in {B}rownian motion indexed by the
  brownian tree.
\newblock {\em arXiv:1811.02825}, 2018.

\bibitem{L76}
D.~L\'{e}pingle.
\newblock La variation d'ordre {$p$} des semi-martingales.
\newblock {\em Z. Wahrscheinlichkeitstheorie und Verw. Gebiete},
  36(4):295--316, 1976.

\bibitem{M18}
Cyril Marzouk.
\newblock On scaling limits of planar maps with stable face-degrees.
\newblock {\em ALEA Lat. Am. J. Probab. Math. Stat.}, 15(2):1089--1122, 2018.

\bibitem{M14}
Gr\'egory Miermont.
\newblock Aspects of random maps.
\newblock {\em Lecture Notes of the 2014 Saint-Flour Probability Summer School.
  Preliminary draft:
  http://perso.ens-Lyon.fr/gregory.miermont/coursSaint-Flour.pdf}.

\bibitem{M13}
Gr\'{e}gory Miermont.
\newblock The {B}rownian map is the scaling limit of uniform random plane
  quadrangulations.
\newblock {\em Acta Math.}, 210(2):319--401, 2013.

\bibitem{MS15}
Jason Miller and Scott Sheffield.
\newblock An axiomatic characterization of the brownian map.
\newblock {\em arXiv:1506.03806}, 2015.

\bibitem{P09}
J.~C. Pardo.
\newblock The upper envelope of positive self-similar {M}arkov processes.
\newblock {\em J. Theoret. Probab.}, 22(2):514--542, 2009.

\bibitem{PS16}
Pierre Patie and Mladen Savov.
\newblock Bernstein-gamma functions and exponential functionals of {L}\'{e}vy
  processes.
\newblock {\em Electron. J. Probab.}, 23:Paper No. 75, 101, 2018.

\bibitem{R12}
V\'ictor Rivero.
\newblock Tail asymptotics for exponential functionals of {L}\'evy processes:
  the convolution equivalent case.
\newblock {\em Ann. Inst. Henri Poincar\'e Probab. Stat.}, 48(4):1081--1102,
  2012.

\bibitem{S17}
Quan Shi.
\newblock Growth-fragmentation processes and bifurcators.
\newblock {\em Electron. J. Probab.}, 22:25 pp., 2017.

\bibitem{S15}
Zhan Shi.
\newblock {\em Branching random walks}, volume 2151 of {\em Lecture Notes in
  Mathematics}.
\newblock Springer, Cham, 2015.
\newblock Lecture notes from the 42nd Probability Summer School held in Saint
  Flour, 2012, \'Ecole d'\'Et\'e de Probabilit\'es de Saint-Flour. [Saint-Flour
  Probability Summer School].

\end{thebibliography}
		
\end{document}